\documentclass[12pt]{article}

\usepackage{amssymb,amsmath}
\usepackage{amsbsy}
\usepackage{amsthm}
\usepackage{mathrsfs}
\usepackage{verbatim}
\usepackage[colorlinks]{hyperref}
\usepackage{fullpage}
\usepackage{mathdots}
\usepackage{graphicx,subfigure}
\usepackage[english]{babel}
\usepackage[utf8]{inputenc}
\usepackage{tikz}

\usetikzlibrary{backgrounds,fit,decorations.pathreplacing}
\usetikzlibrary{positioning}
\usetikzlibrary{calc,through,chains}
\usetikzlibrary{arrows,shapes,snakes,automata, petri}

\usepackage{authblk}

\theoremstyle{definition}
\newtheorem{theorem}{Theorem}
\newtheorem{corollary}{Corollary}

\newtheorem{proposition}{Proposition}

\theoremstyle{definition}
\newtheorem{definition}[equation]{Definition}
\newtheorem{example}[equation]{Example}

\theoremstyle{remark}

\newtheorem{remark}[equation]{Remark}

\numberwithin{equation}{section}
\numberwithin{figure}{section}

\newcommand{\A}{{\mathbb A}}

\newcommand{\C}{{\mathbb C}}
\newcommand{\Z}{{\mathbb Z}}
\newcommand{\Q}{{\mathbb Q}}
\newcommand{\R}{{\mathbb R}}

\newcommand{\N}{{\mathbb N}}

\newcommand{\mc}[1]{\mathcal{#1}} 
\newcommand{\ms}[1]{\mathscr{#1}} 
\newcommand{\mbf}[1]{\mathbf{#1}} 

\newcommand{\mt}[1]{\text{#1}}

\DeclareMathOperator{\Spec}{Spec}

\begin{document}

\title{Classification of 
Reductive Monoid Spaces
Over an Arbitrary Field}
\author[1]{Mahir Bilen Can}
	\date{August 16, 2018}
	\maketitle

\abstract{In this semi-expository paper,
we review the notion of a 
spherical space. 
In particular, we present
some recent results of Wedhorn 
on the classification of 
spherical spaces over arbitrary
fields. As an application,
we introduce and classify 
reductive monoid spaces over an
arbitrary field.}

\section{Introduction}
\label{sec:1}

The classification of 
spherical homogenous 
varieties, more generally
of spherical actions, is
an important, lively, and very 
interesting chapter 
in modern algebraic geometry.
It naturally encompasses 
the classification theories of 
toric varieties, horospherical
varieties, symmetric varieties, 
wonderful compactifications, 
as well as that 
of reductive monoids. 
Our goal in this paper
is to give an expository 
account of some recent 
work in this field. 
As far as we are aware, 
the broadest context in which 
a classification of such objects is achieved,
is in the theory of algebraic
spaces. This accomplishment 
is due to Wedhorn~\cite{Wedhorn}.
Here, we will follow 
Wedhorn's footsteps closely 
to derive some conclusions. 
At the same time, our intention is to provide
enough detail to make 
basic definitions accessible 
to beginners. We will explain a 
straightforward application 
of Wedhorn's progress to 
monoid schemes.

It is not completely 
wrong to claim that 
the origins of our story
go back to Legendre's
work~\cite{Legendre}, 
where he analyzed the 
gravitational potential 
of a point surrounded 
by a spherical surface 
in 3-space. To describe
the representative 
functions of his 
enterprise, he found 
a clever change 
of coordinates 
argument and 
introduced the set of 
orthogonal polynomials 
$\{P_n(x)\}_{n\geq0}$ via  
$
\frac{ 1} {\sqrt{1-2hx +h^2}} 
= \sum_{n\geq 0} P_n(x) h^n,
$
which are now known as
Lagrange polynomials. 
It has eventually been understood 
that the $P_n(x)$'s are the 
eigenfunctions of the operator 
$\varDelta=\frac{\partial^2}{\partial x^2} 
+\frac{\partial^2}{\partial y^2} 
+\frac{\partial^2}{\partial z^2}$ 
restricted to the space 
of $\ms{C}^\infty$ functions 
on the unit 2-sphere.
Nowadays, bits and pieces 
of these elementary facts 
can be found in every 
standard calculus textbook
but their generalization
to higher dimensions 
can be explained in a conceptual 
way by using 
transformation groups.

Let $n$ be a positive
integer, and let $Q_n$ denote the 
standard quadratic form on $\R^n$,
$$
Q_n(x) :=x_1^2+\cdots + x_n^2, 
\ \ x=(x_1,\dots, x_n)\in \R^n.
$$ 
The orthogonal group,
denoted by $\textbf{O}(Q_n)$, 
consists of linear transformations 
$L:\R^n\rightarrow \R^n$ 
such that $Q_n(L(x))=Q_n(x)$ 
for all $x\in \R^n$.
It acts transitively on the $n-1$
sphere 
$S^{n-1}=\{ x\in \R^n :\ Q_n(x)=1\}$,
and the isotropy subgroup in
$\textbf{O}(Q_n)$
of a point $x\in S^{n-1}$ 
is isomorphic to 
$\textbf{O}(Q_{n-1})$.
It is not difficult to write 
down a Lie group automorphism 
$\sigma : \textbf{O}(Q_n)
\rightarrow \textbf{O}(Q_n)$ 
of order two such that 
the fixed point subgroup  
$\textbf{O}(Q_n)^\sigma$
is isomorphic to $\textbf{O}(Q_{n-1})$.
In other words, $S^{n-1}$ has the 
structure of a {\em symmetric space}, 
that is, a quotient manifold of the 
form $G/K$, where $G$ is a 
Lie group and $K=\{g\in G:\ \sigma(g)=g\}$ 
is the fixed subgroup of an
automorphism $\sigma : G\rightarrow G$
with $\sigma^2=id$. 
It is known that the Laplace-Beltrami 
operator $\varDelta_n :=
\sum^n_{i=1} \frac{\partial^2}{\partial x_i^2}$ 
generates the algebra of 
$\textbf{O}(Q_n)$-invariant 
differential operators on $S^{n-1}$.
Moreover, for each $k\in \N$, there is one 
eigenspace $E_k$ of 
$\varDelta_n$ with eigenvalue $-k(k+n-2)$; 
each $E_k$ defines a finite dimensional
and irreducible representation of 
$\textbf{O}(Q_n)$.
In addition, 
the Hilbert space of square integrable 
functions on $S^{n-1}$ has 
an orthogonal space 
decomposition, 
$L^2(S^{n-1}) = \sum_{k=0}^\infty E_k$.
The last point of this
example is the most important 
for our purposes; the representation 
of $\textbf{O}(Q_n)$ on 
the polynomial functions on $S^{n-1}$ is 
multiplicity-free! 
All these facts are well known and 
can be found in classical textbooks 
such as~\cite{Helgason:GGA} or
\cite{BrockertomDieck}.

We will give another example 
to indicate how often we 
run into such multiplicity-free
phenomena in the theory of Lie groups. 
This time we start with an
arbitrary compact Lie group, 
denoted by $K$, and consider 
$C^0(K,\C)$, the algebra of 
continuous functions
on $K$ with complex values. 
The doubled group $K\times K$
acts on $K$ by translations:
$$
(g,h) \cdot x = gxh^{-1} \text{ for all } g,h,x\in K.
$$
In particular, we 
have a representation 
of $K\times K$ 
on $C^0(K,\C)$,
which is infinite dimensional 
unless $K$ is a finite group. 
Let $\ms{F}(K,\C)$ denote
the subalgebra of {\em 
representative functions},
which, by definition, are 
the functions  
$f\in C^0(K,\C)$ 
such that $K\times K\cdot f$
is contained in a finite dimensional
submodule of $C^0(K,\C)$. 
The theorem of 
Peter and Weyl 
states that 
$\ms{F}(K,\C)$
is dense in $C^0(K,\C)$.
Moreover as a
representation of $K\times K$, 
$\ms{F}(K,\C)$ has
an orthogonal space decomposition
into finite dimensional 
irreducible 
$K\times K$-representations,
each irreducible occurring with 
multiplicity at most one.

We have a quite related, analogous 
statement on the multiplicity-freeness 
of the $G\times G$-module 
structure of the 
coordinate ring $\C[G]$
of a reductive complex 
algebraic group $G$. 
Indeed, it is a well known 
fact that on every compact 
Lie group $K$ there exists 
a unique real algebraic 
group structure, and furthermore,
its complexification $K(\C)$
is a complex algebraic group
which is reductive. 
Conversely, any reductive 
complex algebraic group
has an algebraic
compact real form
and this real form 
has the structure of a
compact Lie group. 
Two compact Lie groups
are isomorphic as 
Lie groups 
if and only if 
the corresponding reductive
complex algebraic groups
are isomorphic (see~\cite{Chevalley, OnishchikVinberg}).

The unifying theme of these  
examples, as we alluded to before, 
is the multiplicity-freeness 
of the action on the function space
of the underlying variety or manifold. 
It turns out that the
multiplicity-freeness 
is closely related to the size 
of orbits of certain subgroups.
To explain this more clearly, 
for the time being, 
we confine ourselves to the 
setting of affine algebraic varieties
that are defined over $\C$. 
But we have a disclaimer:
irrespective of the underlying 
field of definitions, our tacit assumption 
throughout this paper is that
all reductive groups are connected
unless otherwise noted.

Now, let $G$ be a reductive complex
algebraic group, and let $X$ be
an affine variety on which $G$
acts algebraically. 
We fix a Borel subgroup $B$, that is, 
a maximal connected solvable 
subgroup of $G$. It is well known that 
every other Borel subgroup of 
$G$ is conjugate to $B$. 
The action of $G$ on $X$ gives
rise to an action of $G$, 
hence of $B$, on the coordinate 
ring $\C[X]$. 
Let us assume that $B$ has
finitely many orbits in $X$, so,
in the Zariski topology,  
one of the $B$-orbits is open.
Let $\chi$ be a character
of $B$ and $x_0$ be a 
general point from the dense $B$-orbit. 
Let $f$ be a regular 
function that is only defined
on the open orbit. Hence, 
we view $f$ as an element of 
$\C(X)$, the field of rational 
functions on $X$. If $f$ 
is an $\chi$-eigenfunction for the 
$B$-action, that is, 
$$
b\cdot f = \chi(b) f \ \text{ for all } b\in B,
$$
then the value of $f$ on the
whole orbit is uniquely determined
by $\chi$ and the base point
$x_0$. Indeed, any point
$x$ from the open orbit 
has the form $x=b\cdot x_0$ 
for some $b\in B$, 
therefore, 
$$
f(x) = b^{-1}\cdot f(x_0) = \chi(b^{-1}) f(x_0).
$$
This simple argument shows 
that there exists at most one 
$\chi$-eigenvector in $\C[X]$
whose restriction to the open 
$B$-orbit equals $f$. 
As irreducible representations of 
reductive groups are parametrized
by the `highest' $B$-eigenvectors, 
now we understand that 
the number of occurrence 
of the irreducible
representation corresponding 
to the character $\chi$ in $\C[X]$ 
cannot exceed 1. In other words, 
$\C[X]$ is a multiplicity-free $G$-module. 
The converse of this statement is 
true as well; if a linear and 
algebraic action
$G\times \C[X]\rightarrow \C[X]$
is multiplicity-free, then 
a Borel subgroup of $G$ 
has finitely many orbits in $X$,
and one of these orbits is open 
and dense in $X$ 
(see~\cite{Brion:Quelques}, as well 
as \cite{Vinberg:Complexity, Popov:Contractions}).
This brings us to the 
(special case of a) fundamental 
definition that will occupy us 
in the rest of our paper.

\begin{definition}
Let $k$ be an algebraically 
closed field, and let $G$ be a 
reductive algebraic group
defined over $k$. 
Let $X$ be a normal variety
that is defined over $k$, and 
finally, let $\psi:G\times X\rightarrow X$
be an algebraic action of $G$ on $X$. 
If the restriction of the 
action to a Borel subgroup 
has finitely many orbits, then
the action is called {\em spherical}, 
and $X$ is called 
a {\em spherical $G$-variety}.
\end{definition}

Let $H\subset G$
be a closed subgroup 
in a reductive group $G$.
The homogenous space
$G/H$ is called 
spherical if $BH$ 
is an open dense subvariety
in $G$ for some Borel subgroup
$B\subset G$. 
A $G$-variety $X$ is 
called an equivariant embedding
of $G/H$ if $X$ has an open
orbit that is isomorphic to $G/H$.
In particular, spherical varieties 
are normal $G$-equivariant
embeddings of spherical 
homogenous $G$-varieties.
To see this, let $x_0$ be
a general point from the 
open $B$-orbit in a spherical
variety $X$, and let $H$
denote the isotropy subgroup
$H=\{g\in G:\ g\cdot x_0 =  x_0 \}$ in $G$.  
Clearly, $G/H$ is isomorphic
to the open $G$-orbit, and 
it is a spherical subvariety of $X$. 
It follows that $X$ is a $G$-equivariant 
embedding of $G/H$.

It goes without saying that 
this theory, as we know it,
owes its existence
to the work of Luna and 
Vust~\cite{LunaVust}
who classified the equivariant 
embeddings of spherical 
homogenous varieties over 
algebraically closed fields of 
characteristic 0.
Their results are extended to
all characteristics 
by Knop in~\cite{Knop89}.
What is left is the classification 
of spherical subgroups (over 
arbitrary fields) 
and this program is well on its way;
see the recent paper~\cite{BraviPezzini16}
and the references therein. 
For a good and broad introduction
of the field of equivariant embeddings, 
up to 2011, 
we recommend the encyclopedic
treatment~\cite{Timashev}
of Timashev.

We mentioned in the first 
paragraph of this introduction that the examples
of spherical embeddings 
include algebraic monoids. 
By definition, an algebraic 
monoid over an algebraically
closed field is an algebraic 
variety $M$
endowed with an associative 
multiplication morphism 
$m: M\times M\rightarrow M$ 
and there is a neutral element
for the multiplication. 
The foundations of these monoid 
varieties are secured mainly by the 
efforts of Renner and Putcha,
who chiefly developed the 
theory for linear (affine) algebraic monoids
(see~\cite{Renner,Putcha}).
From another angle, 
Brion and Rittatore 
looked at the general structure
of an algebraic monoid. 
Amplifying the importance 
of linear algebraic monoids,  
Brion and Rittatore showed that 
any irreducible normal algebraic 
monoid is a homogenous 
fiber bundle over an abelian 
variety, where the fiber over 
the identity element is 
a normal irreducible linear algebraic monoid 
(see Brion's lecture notes~\cite{Brion:FieldsLectures}).
In this regard, let us recite a
result of Mumford about the 
possible monoid structure 
on a complete irreducible variety
(\cite[Chapter~II]{Mumford}): 
if a complete irreducible 
variety $X$ has a (possibly
nonassociative) composition
law $m:X\times X\rightarrow X$ 
with a neutral element,
then $X$ is an abelian variety
with group law $m$.
In other words, an irreducible 
and complete algebraic monoid 
is an abelian variety. 
This interesting result of 
Mumford is extended to families 
by Brion in~\cite{Brion:FieldsLectures}.

A $G$-equivariant embedding
is said to be simple if it has a 
unique closed $G$-orbit. 
A {\em reductive monoid} is 
an irreducible algebraic monoid
whose unit-group is a reductive group. 
The role of such monoids
for the theory of 
equivariant embeddings 
was understood very early;
Renner recognized in
~\cite{Renner:ClassificationVarieties} and
\cite{Renner:vonNeumann}
that the normal reductive monoids 
are simple $G\times G$-equivariant
embeddings of reductive groups.
Explicating this observation, 
Rittatore showed in~\cite[Theorem 1]{Rittatore98}
that every irreducible 
algebraic monoid $M$ is 
a simple $G(M)\times G(M)$-equivariant 
embedding of its unit-group $G(M)$.
In the same paper, by using Knop's 
work on colored fans, which we will
describe in the sequel, 
Rittatore described
a classification of reductive monoids
in terms of colored cones. 
This classification is a 
generalization of the earlier classification of 
the reductive normal monoids by
Renner~\cite{Renner:ClassificationVarieties} 
and Vinberg~\cite{Vinberg:ReductiveSemigroups}.

Let $k$ be an algebraically
closed field. We  
define a {\em toric variety over $k$} 
as a normal algebraic variety
on which the torus $(k^*)^n$ acts 
(faithfully) with an open orbit. 
Toric varieties are prevalent
in the category of spherical varieties
in the following sense: if $X$
is a spherical $G$-variety over $k$, 
then the closure in $X$ of the
$T$-orbit of a general point
from the open $B$-orbit will 
be a toric variety. Here, $T$ stands 
for a maximal torus contained in $B$. 
Also, let us not forget the fact that 
the affine toric varieties are precisely
the commutative reductive monoids, 
\cite[Theorem 3]{Rittatore98}.

Now let $k$ be an arbitrary field, 
and let $T$ be a torus that is isomorphic 
to $(k^*)^n$, where both the 
isomorphism and the torus $T$ 
are defined over $k$. The technical term
for such a torus is {\em $k$-split}
or {\em split torus over $k$}.
Toric varieties defined by 
split tori are parametrized by 
combinatorial objects, 
the so-called fans. This fact
is due to Demazure in the 
smooth case~\cite{Demazure} 
and Danilov
for all toric varieties
\cite{Danilov}.
A {\em fan} $\mc{F}$ in $\Q^n$ is 
a finite collection of strictly convex
cones such that 
1) every face of an element 
$\mc{C}$ from $\mc{F}$ lies
in $\mc{F}$;
2) the intersection
of two elements of $\mc{F}$ 
is a face of both of the cones. 
Here, 
by a {\em cone} we mean 
a subset of $\Q^n$ that 
is closed under
addition and 
the scaling action
of $\Q_{\geq 0}$.
A {\em face} of a 
cone $\mc{C}$ 
is a subset 
of the form 
$\{v\in \mc{C} :\ 
\alpha (v)=0\}$, 
where $\alpha$ is 
a linear functional
on $\Q^n$ 
that takes nonnegative
values on $\mc{C}$. 
Let us define two more 
notions that will be used 
in the sequel. 
A cone is called {\em strictly 
convex}, if it does not
contain a line.
The {\em relative interior} 
of $\mc{C}$, denoted
by $\mc{C}^0$, is what is left
after removing all of 
its proper faces.

The toric varieties defined by 
nonsplit tori are quite  
interesting and their classification
is significantly more intricate.
Parametrizing 
combinatorial objects in this case,
as shown by Huruguen in 
~\cite{Huruguen} are 
fans that are stable 
under the Galois group of 
a splitting field. To explain this,
we extend our earlier
definition of toric varieties
as follows.
Let $k$ be a field, and let $\overline{k}$
denote an algebraic closure of $k$. 
Let $T$ be a torus defined
over $k$. A normal $T$-variety 
$Y$ is called a {\em toric variety over $k$} if 
$Y(\overline{k})$
is a toric variety with respect to 
$T(\overline{k})$.
(This notation will be made precise 
in Section~\ref{S:Notation}.)
Let us continue with the assumption 
that $Y$ is a toric variety with respect 
to a $k$-split torus $T$, and 
let $k'\subseteq k$ be a field 
extension with $T$ not 
necessarily split over $k'$.   
Let $\varGamma$
denote the Galois group
of the extension. 
Since all tori become split
over a finite separable field 
extension, we will assume 
also that 
$k'\subset k$ is a finite extension.
Of course, it may happen
that $Y$ is not defined over $k'$. 
If it is defined, then $Y(k')$
is called the(!) $k'$-form
of $Y$.
In~\cite[Theorem 1.22]{Huruguen}, 
Huruguen gives two necessary and 
sufficient conditions for
the existence of a $k'$-form.
The first of these 
two conditions is rather 
natural in the sense 
that the fan of $Y(k)$ is stable under 
the action of $\varGamma$. 
The second condition
is also concrete, 
however, it is more difficult to 
check. 
Actually, it is a criterion about 
the quasiprojectiveness
of an equivariant embedding
in terms of the fan of the embedding.
Its colored version, namely 
a quasiprojective colored fan,
which is also used by Huruguen,
is introduced by Brion in~\cite{Brion:Duke}.
We postpone the precise 
definition of a quasiprojective colored
fan to Section~\ref{SS:Quasiprojective},
but let us mention that, 
in~\cite{Huruguen}, 
Huruguen shows by an example
that quasiprojectiveness is an essential
requirement for $Y$ to have 
a $k'$-form.

Using the underlying idea that 
worked for toric varieties, in the 
same paper, Huruguen proves more.
Let $Y$ be a spherical homogenous
variety for a reductive group $G$ 
defined over a perfect field $k$. 
(The perfectness assumption is 
a minor glitch; it will not be needed  
once we start to work with the 
algebraic spaces as Wedhorn does.
Indeed, it is required here so that
the homogenous varieties have rational points.)
Let $\overline{k}$ be an 
algebraic closure of $k$, and 
let $\varGamma$ denote
the Galois group of $k\subset \overline{k}$. 
The introduction of 
the absolute Galois group,
which is often too big, 
is not a serious problem since
in the situations that we are interested in, 
the absolute Galois group factors 
through a finite quotient.
Let $\overline{Y}$ be a $G$-equivariant
embedding of a spherical homogenous
$G$-variety $Y$ that is defined over $\overline{k}$.
In \cite[Theorem 2.26]{Huruguen}, Huruguen shows
that a $G$-equivariant embedding
$\overline{Y}$ of $Y$
has a $k$-form if and only if 
the colored fan of $\overline{Y}(\overline{k})$
is $\varGamma$-stable and 
it is ``quasiprojective with respect to 
$\varGamma$.'' 
Once again, Huruguen shows by 
examples that the failure of 
the second condition implies
the nonexistence of $k$-forms.

As Wedhorn shows in~\cite{Wedhorn} 
via algebraic spaces, 
as soon as we get over 
the contrived emotional barrier set 
in front of us by algebraic varieties, 
we happily see that 
the existential questions (about $k$-forms) 
disappear. 
Roughly speaking, in some sense, algebraic spaces are 
to schemes, 
what schemes are to 
algebraic varieties. 
Such is the transition
from spherical
varieties to spherical 
spaces.

\begin{definition}\label{D:Wedhorns}
For an arbitrary base 
scheme $S$ and 
a reductive group scheme $G$ 
over $S$, 
a spherical $G$-space
over $S$ is a flat
separated algebraic space 
of finite presentation over $S$
with a $G$-action such that 
the geometric fibers 
are spherical varieties. 
\end{definition}

Notice that we have not 
defined reductive $S$-groups yet.
But timeliness may not be the only problematic
aspect of Definition~\ref{D:Wedhorns}. Understandably, 
it may look overly general
at first sight. Nevertheless, 
the definition has many remarkable 
consequences.
For example, according to this 
definition, the property for a flat finitely presented 
$G$-space with normal geometric fibers
to be spherical is open and 
constructible on the base scheme. 
Moreover, for a flat finitely
presented subgroup scheme $H$ 
of $G$ the property to be spherical 
is open and closed on the base scheme. 
Most importantly for our purposes, 
the specialization of the base scheme
to the spectrum of a field in 
Definition~\ref{D:Wedhorns} 
yields a general classification of 
spherical $G$-spaces over arbitrary fields
in terms of colored fans that 
are stable under the Galois group.
In fact, for spaces over fields,
Wedhorn's definition is much easier 
to use.

Since explaining the concepts associated 
with spherical algebraic spaces and colored
fans will take a bulky portion of our paper, 
we postpone giving the ballistics of 
Wedhorn's theory to 
Section~\ref{SS:WedhornsClassification}.
Nearing the end of our lengthy introduction,
let us mention that as a rather straightforward
consequence of Wedhorn's deus ex machina 
we obtain a classification of ``reductive monoid spaces,''
the only new definition that we offer in this paper.  

The organization of our paper is 
as follows. In Section~\ref{S:Notation},
we introduce what is required to explain 
spherical algebraic spaces in the following
order: 
Subsection~\ref{SS:Schemes} is on the fundamentals; 
we introduce the notion of scheme as a functor.
Subsection~\ref{SS:AlgebraicSpaces} sets up
the notation for algebraic spaces.
In Subsection~\ref{SS:GroupSchemes},
we talk about group schemes, which 
is followed by Subsection~\ref{SS:ReductiveGroups}
on reductive group schemes. 
The Subsection~\ref{SS:ParabolicSubgroups}
is devoted to parabolic subgroups.
In Subsection~\ref{SS:Actions},
we briefly discuss actions of group schemes.
In Section~\ref{S:AlgebraicMonoids}, we
discuss affine monoid schemes 
and prove a souped-up version of a 
result of Rittatore on the unit-dense 
algebraic monoids. 
The beginning of Section~\ref{S:FromSpherical}
is devoted to the review of colored fans.
In Subsection~\ref{SS:Quasiprojective},
we review the quasiprojectiveness 
criterion of Brion and mention Huruguen's
theorem on the $k$-forms of spherical
varieties over perfect fields. 
In Subsection~\ref{SS:WedhornsClassification},
we present Wedhorn's generalization
of Huruguen's results. In particular
we talk about his colored fans 
for spherical algebraic spaces. 
In the subsequent Section~\ref{S:ReductiveMonoids},
we introduce reductive monoid spaces 
and apply Wedhorn's and Huruguen's 
theorems.
Finally, we close our paper 
in Subsection~\ref{SS:LinedClosure} by 
presenting an application of our observations 
to the lined closures of representations
of reductive groups, whose 
geometry is investigated by 
De Concini in~\cite{DeConcini}.
\\

\textbf{Acknowledgement.}
I thank the organizers of the 
2017 Southern Regional Algebra Conference:
Laxmi Chataut, J\"org Feldvoss, Lauren Grimley, 
Drew Lewis, Andrei Pavelescu, and Cornelius Pillen.
I am grateful to J\"org Feldvoss, Lex Renner, Soumya Dipta Banerjee,
and to the anonymous referee for their very careful reading of the paper
and for their suggestions, which improved the quality of the article.
This work was partially supported by a grant from 
the Louisiana Board of Regents.

\section{Notation and preliminaries}\label{S:Notation}

The purpose of this section
is to set up our notation and 
provide some background 
on algebraic spaces, 
reductive $k$-groups, and 
reductive $k$-monoids.
We tried to give most of 
the necessary
definitions for explaining 
the logical dependencies. 
For standard algebraic geometry 
facts, we recommend
the books~\cite{DemazureGabriel} 
and~\cite{Knutson}.
(It seems to us that the 
Stacks Project~\cite{StacksProject}
will eventually replace all standard
references.)
In addition, we find that Brion's lecture notes
in~\cite{Brion:CliffordLectures}  
are exceptionally valuable as a resource
for the background on algebraic groups.
\vspace{.5cm}

{\em Notation: Throughout our paper 
$k$ will stand for a field, 
as $G$ does for a group.
$k$ is called a perfect field if 
every algebraic extension 
of $k$ is separable. 
}

\subsection{Schemes.}\label{SS:Schemes}

{\it 
A point that we want to make 
in this section is that the only 
way to know a group scheme 
is to know all (affine) group 
schemes related to it. In fact,
this statement is a theorem.}

\vspace{.25cm}

{\em Terminology:} 
Let $(X,\ms{F})$ be a pair of a 
topological space $X$ 
and a sheaf of 
rings $\ms{F}$ on $X$. If for each point 
$x\in X$, the corresponding stalk $\ms{F}_{x}$
is a local ring, then the pair 
$(X,\ms{F})$ is called a {\em locally ringed space}. 
An {\em affine scheme} is a locally ringed 
space which is of the form $(\Spec (R),\ms{O}_{\Spec (R)})$,
where $R$ is a commutative ring, and $\Spec (R)$ 
is its spectrum endowed with the Zariski topology. 
The sheaf $\ms{O}_{\Spec (R)}$ is the structure sheaf of $\Spec (R)$. 
A {\em scheme} is a locally ringed space $(X,\ms{O}_X)$ which 
has a covering by open subsets $X=\bigcup_{i\in I} U_i$ 
such that each pair $(U_i, \ms{O}_X \vert_{U_i})$
is an affine scheme.


There is a tremendous advantage 
of using categorical language while
studying a scheme in relation with 
others. Therefore, we proceed with  
the identification of schemes with their 
functors of points. In this regard, 
we will use the following standard 
notation throughout: 
\begin{table}[h]
\begin{tabular}{lll}
$\textbf{Obj}(\mc{C})$ &: & the class of 
objects of a category $\mc{C}$\\
$\textbf{Mor}(\mc{C})$ &: & the class of 
morphisms between the objects of $\mc{C}$\\
$\text{Mor}(X,Y)$ &: & the set of morphisms from $X$ to $Y$, where 
$X,Y\in \textbf{Obj}(\mc{C})$\\
$\mc{C}^{o}$  &:& the category opposite to $\mc{C}$\\
$\textbf{(schemes)}$ &: & the category of schemes \\
$\textbf{(sets)}$ &:& the category of sets 
\end{tabular}
\end{table}


The {\em functor of points} 
of a scheme $X$ is the functor 
$h_X : \textbf{(schemes)}^{o} \rightarrow \textbf{(sets)}$
that is defined by the following assignments:
\begin{enumerate}
\item if $Y \in \textbf{Obj}( \textbf{(schemes)}^{o})$, then 
$h_X(Y)=\mt{Mor}(Y,X)$;
\item if $f\in \mt{Mor}(Y,Z)$, then $h_X(f)$ is the set map 
\begin{align*}
h_X(f) :\ h_X(Z) &\longrightarrow h_X(Y)\\
g & \longmapsto f\circ g
\end{align*}
\end{enumerate}
Let $h$ denote the following natural transformation: 
\begin{align}\label{A:as functors}
h : \textbf{(schemes)} & \rightarrow 
\textbf{functors((schemes)$^{o}$, (sets))} \notag \\
X & \mapsto h_X
\end{align}
Here, $\textbf{functors(-,-)}$ stands for
the category whose objects are functors,
and its morphisms are the natural transformations
between functors.
It follows from Yoneda's Lemma that 
(\ref{A:as functors}) is an equivalence
onto a full subcategory of the target category. 
In other words, a scheme $X$ is uniquely
represented by its functor of points $h_X$. 
\vspace{.25cm}

{\em Terminology:} 
In the presence of a 
morphism $Y\rightarrow X$
between two schemes $X$ 
and $Y$, we will occasionally say  
that {\em $Y$ is a scheme over $X$}
or that {\em $Y$ is an $X$-scheme}.
If there is no danger for confusion, 
we will write $X(Y)$ for $h_X(Y)$, 
which is the set of all morphisms from $Y$ to $X$. 
If, in addition, 
$Y$ is an affine scheme of the form $Y=\Spec (R)$,
then we often write $X(R)$ instead of $X(Y)$, 
and we say that {\em $X$ is a scheme over $R$}.
\vspace{.5cm}

We want to show that 
the functor $h$ in (\ref{A:as functors}) behaves well upon restriction
to the category of schemes over an affine scheme.
To this end, let $R$ be a commutative ring, 
and denote the category of $R$-schemes 
by $\textbf{($R$-schemes)}$.
A {\em morphism} in this category is a commutative  
diagram of morphisms as in Figure~\ref{F:morphism}.

\begin{figure}[h]
\begin{center}
\begin{tikzpicture}[scale=1.25]
\begin{scope}[xshift=-2cm,yshift=0cm]
\node at (-1,0.5)  (1) {$X$};
\node at (1,0.5) (2) {$Y$};
\node at (0,-0.5) (3) {$\Spec(R)$};
\draw[->,thick] (1) to (2);
\draw[->,thick] (2) to (3);
\draw[->,thick] (1) to (3);
\end{scope}
\end{tikzpicture}
\caption{A morphism in $\textbf{($R$-schemes)}$.}
\label{F:morphism}
\end{center}
\end{figure}

It is well known that 
the category $\textbf{($R$-schemes)}$
is equivalent to the opposite category of commutative 
$R$-algebras, denoted by $\textbf{($R$-algebras)}^o$.
The highlight of this subsection is the following 
result whose proof easily follows from the definitions.

\begin{proposition}\label{P:as functors}
The functor 
$$
h: \textbf{($R$-schemes)} \longrightarrow 
\textbf{functors (($R$-algebras),(sets))},
$$
which is obtained from 
(\ref{A:as functors}) by restriction 
to the subcategory $\textbf{($R$-schemes)}$,
is an equivalence onto a full subcategory of 
the target category.
In particular, a scheme over $R$ is determined 
by the restriction of its functor of points to the
category of affine schemes over $R$. 
\end{proposition}

For our purposes, the 
most important consequence of 
Proposition~\ref{P:as functors} is 
that an $R$-scheme can be 
thought of as a sheaf of sets 
on the category of $R$-algebras 
with respect to the Zariski topology.
Closely related to this 
sheaf realization of schemes is 
the notion of an ``algebraic space.''
Concisely, and very roughly 
speaking, an algebraic space
is a sheaf of sets on the 
category of $R$-algebras 
with respect to the ``\'etale topology.''
We will give a more precise definition 
of algebraic spaces in the next subsection.
Next, we remind 
ourselves of some basic notions 
regarding the morphisms between
schemes.

\begin{enumerate}

\item 
A ring homomorphism $f: R\rightarrow S$
is called {\em flat} if the 
associated induced functor 
$$f_*:\textbf{($R$-modules)} \rightarrow \textbf{($S$-modules)}$$
is exact, that is to say, it 
maps short-exact sequences to short-exact sequences.
In this case, the morphism 
$f^*: \text{Spec}(S)\rightarrow \text{Spec}(R)$
is called flat. 
More generally, a map $f: X\rightarrow Y$ 
of schemes is called {\em flat}
if the induced functor 
$$
f^*: \textbf{(Quasicoherent sheaves on $Y$)}
\rightarrow \textbf{(Quasicoherent sheaves on $X$)}
$$
is exact.

\item A ring homomorphism $f: R\rightarrow S$ 
is called {\em of finite presentation} if $S$ is isomorphic 
(via $f$) to a finitely generated polynomial algebra over $R$
and the ideal of relations among the generators is finitely generated.
A map 
$f: X\rightarrow Y$ 
of schemes is called {\em a map of finite presentation at $x\in X$} 
if there exists an affine open neighborhood $U=\text{Spec}(S)$ 
of $x$ in $X$ and an affine open neighborhood $V=\text{Spec}(R)$
of $Y$ with $f(U)\subseteq V$ such that the induced ring map $R\rightarrow S$
is of finite presentation. 
More generally, a map 
$f: X\rightarrow Y$ 
of schemes is called {\em of locally finite presentation} if it is of 
finite presentation at every point $x\in X$.

\item A map $f: X\rightarrow Y$
of schemes is called {\em separated} if 
the induced diagonal morphism 
$X\rightarrow X\times_Y X$ is 
a closed immersion.

\item A map $f: X\rightarrow Y$
of schemes is called {\em unramified} if 
the induced diagonal morphism 
$X\rightarrow X\times_Y X$ is 
an open immersion.

\item A map 
$f: X\rightarrow Y$ 
of schemes is called {\em \'etale} if it is 
unramified, flat, and of locally finite presentation. 
\end{enumerate}

Associated with these types of maps, 
we have two important topologies.
These are 
\begin{enumerate}
\item The topology on the category of 
schemes associated with the \'etale
morphisms; this topology is called the {\em \'etale topology.} 
\item The topology on the category of 
schemes associated with the set of maps,
flat and of locally finite presentation;
this topology is called the {\em fppf topology.} 
\end{enumerate}

\subsection{Algebraic spaces.}\label{SS:AlgebraicSpaces}

Let $S$ be a scheme. 
An {\em algebraic space $X$ over $S$} 
is a functor 
$$
X: \textbf{($S$-schemes)}^o \longrightarrow 
\textbf{(sets)}
$$
satisfying the following properties: 
\begin{enumerate}
\item[(i)] $X$ is a sheaf in the fppf topology.
\item[(ii)] The diagonal morphism 
$X\rightarrow X\times_S X$ is representable
by a morphism of schemes.
\item[(iii)] There exists
a surjective \'etale
morphism $\widetilde{X}\rightarrow X$,
where $\widetilde{X}$ is an $S$-scheme.
\end{enumerate}
It turns out that, in this definition, 
specifically in part (i),  
replacing the fppf topology by the \'etale topology on 
the category of schemes does not cause any harm. 
In other words, the resulting functor
describes the same algebraic space 
(see~\cite[\href{http://stacks.math.columbia.edu/tag/0123}{Tag 076L}]{StacksProject}).
Let us mention in passing that 
the category of schemes is a 
full subcategory of the category of algebraic spaces.

We will follow the 
standard assumption,
as in~\cite{Wedhorn} (and 
as in~\cite{Knutson}), that 
all algebraic spaces 
are quasiseparated over 
some scheme, hence
they are Zariski locally 
quasiseparated. This implies
that our algebraic spaces are reasonable and decent 
in the sense of~\cite[\href{http://stacks.math.columbia.edu/tag/0123}{Tag 03I8}]{StacksProject} 
and~\cite[\href{http://stacks.math.columbia.edu/tag/0123}{Tag 03JX}]{StacksProject}.

Our final note in this subsection
is that if $k$ is a separably
closed field, then an algebraic space over $k$ 
is a $k$-scheme. 

\subsection{Group schemes.}\label{SS:GroupSchemes}

Let $R$ be a commutative ring.
A {\em group scheme} over $R$ is an $R$-scheme
whose functor of points factors through 
the forgetful functor from the category 
of groups to the category of sets. 
In a nutshell, an $R$-scheme $G$ 
is called a group scheme over $R$ 
if for every $R$-scheme $S$, 
there is a natural group structure 
on $G(S)$ which is functorial with respect
to the morphisms $R\rightarrow S$.

\vspace{.5cm}

{\em Terminology:} If $G$ is a 
(group) scheme over $R$,
then the morphism $G\rightarrow \Spec(R)$
is called the {\em structure morphism}.

\begin{definition}
Let $k$ be a field and let $G$ 
be a group scheme over $k$. 
We will call $G$ a {\em $k$-group} 
(or, an {\em algebraic group over $k$}) 
if $G$ is of finite type as a scheme
over $k$. 
If $k'$ is a subfield 
of $k$, then $G$ is said to be 
{\em defined over $k'$} if $G$,
as a scheme, and all of 
its group operations
as morphisms are defined
over $k'$. 
\end{definition}

Let $S$ be a scheme, and let $G$ 
be a group scheme over $S$. 
$G$ is called {\em affine, smooth, flat,} or  
{\em separated}, respectively, if 
the structure morphism $G\rightarrow S$
is affine, smooth, flat, or separated,
respectively. 
We have some remarks 
regarding these properties:

\begin{enumerate}
\item Any $k$-group $G$ is separated 
as an algebraic scheme. Indeed,
the diagonal in $G\times G$ is closed
as being the inverse image of 
the ``identity'' point of $G(k)$  
under the morphism 
$t: G\times G \rightarrow G$
defined by 
$t(g,h)= gh^{-1}$.

\item Let $k$ be a perfect field, and let 
$G$ be a $k$-group scheme.
If the underlying scheme of $G$ 
is reduced, that is to say, the 
structure ring has no nilpotents,
then $G$ is smooth (see~\cite[Pg 287]{DemazureGabriel}).

\item 
The Cartier's theorem states that if the characteristic
of the underlying field $k$ is 0,
then a $k$-group scheme is reduced 
(see~\cite[Pg 101]{Mumford}).
It follows that, in characteristic 0,
all $k$-groups are smooth.

\end{enumerate}

\vspace{.5cm}

Next, we will define a particular $k$-group
that is of fundamental importance
for the whole development of algebraic
group theory.

\begin{example}
Let $V$ be a finite dimensional 
vector space over $k$. 
The {\em general linear group}, 
denoted by $GL(V)$, 
is the group functor such that 
$$
GL(V) (S) = \text{ the automorphism 
group of the sheaf of 
$\mc{O}_S$-modules 
$\mc{O}_S \otimes_k V$},
$$
where $S$ is a $k$-scheme. 
By choosing a basis 
for $V$, we see that $GL(V)(S)$ is 
isomorphic to the group of invertible 
$n\times n$ matrices with 
coefficients in the algebra $\mc{O}_S(S)$, 
hence, $GL(V)$ is 
represented by the open affine scheme 
$$
\mt{GL}_n=\{P\in \A^{n^2} :\ \det P \neq 0 \}.
$$
It follows that $GL(V)$ is 
smooth and connected. 
\end{example}

\begin{definition}
A group scheme $G$ is called {\em linear} if it is 
isomorphic to a closed subgroup scheme 
of $\mt{GL}_n$ for some positive integer $n$.
If $V$ is a vector space, then a homomorphism 
$\rho : G\rightarrow GL(V)$, which is a morphism
of schemes, is called a {\em linear 
representation} of $G$ on $V$.  
In this case, $V$ is called a {\em $G$-module}.
\end{definition}

Note that if a $k$-group scheme is linear, 
then it is an affine scheme. 
Note conversely that, every $k$-group scheme 
which is affine and of finite type is a linear group scheme 
(see~\cite[Proposition 3.1.1]{Brion:CliffordLectures}).

\vspace{.5cm}

In some sense, atomic pieces of 
$k$-groups are given by the following 
two very special $k$-groups:

\begin{example}
Let $+$ and $\cdot$ denote the 
addition and the multiplication
operations on the field $k$.  
\begin{itemize}
\item The additive 1-dimensional 
$k$-group, denoted by $\mbf{G}_a$, is the 
affine line $\mathbb{A}^1_k$ considered with the 
group structure $(k,+)$. 
\item The multiplicative 1-dimensional
$k$-group, denoted by $\mbf{G}_m$, 
is $\mathbb{A}_k^1 - \{0\}$ considered 
with the group structure $(k^*,\cdot)$. 
\end{itemize}
\end{example}

\begin{definition}
A $k$-group $G$ is called 
a torus if there exists an isomorphism 
$\zeta: G \rightarrow \mbf{G}_m\times 
\cdots \times \mbf{G}_m$
($n$ copies, for some $n\geq 1$) 
over some field containing $k$. 
Assuming that $G$ is a torus defined over $k'$,
and that $k'$ is a subfield of $k$, $G$ is called $k$-split 
if both of $G$ and $\zeta$ are defined over $k$. 
\end{definition}

A modern proof of the following 
basic result can be found in
~\cite[Appendix A]{Conrad:Reductive}.

\begin{theorem}[Grothendieck]
Let $k'$ and $k$ be two fields such that $k'\subset k$. 
Let $G$ be a smooth connected 
affine $k$-group. If $G$ is defined over $k'$, 
then $G$ contains a maximal $k'$-split torus $T$
such that $T(\overline{k})$ is a maximal torus of $G(\overline{k})$. 
\end{theorem}

\subsection{Reductive group schemes.}\label{SS:ReductiveGroups}

{\em Almost any fact about 
reductive group schemes 
can be found in Conrad's 
SGA3 replacement 
~\cite{Conrad:Reductive}.}

\vspace{.5cm}

Let $G$ be a $k$-group, and let 
$\rho : G \hookrightarrow GL(V)$ 
be a finite dimensional faithful 
linear representation. 
An element $g$ from $G$ 
is called {\em semisimple} 
(respectively {\em unipotent}) 
if the linear operator $\rho(g)$ 
on $V$ is diagonalizable 
(respectively, unipotent). 
It is not difficult to check that 
these notions (semisimplicity 
and unipotency) are independent 
of the faithful representation,
and they are preserved under 
$k$-homomorphisms. 
Therefore, the following
 definition is unambiguous:

\begin{definition}
Let $G$ be a linear algebraic $k$-group. 
A $k$-subgroup $U$ of $G$ is called 
{\em unipotent} if every element of $U$ is unipotent. 
\end{definition}

Next, we give the definition of 
a reductive group. 
However, we will do this 
in the opposite of the
chronological development
of the subject to emphasize 
the differences. So,  
let us start with the 
definition of the relative reductive
group schemes. This is 
most useful for studying 
properties that are preserved 
in families over a commutative 
ring.

\vspace{.5cm}

Let $S$ be a scheme, 
and let $G$ be a smooth 
$S$-affine group scheme over $S$. 
Let $s$ be a point from $S$, and 
denote by $G_{\overline{s}}$ 
the geometric fiber
$G\times_S \Spec(\overline{k(s)})$ 
of $G\rightarrow S$.  
Here, $k(s)$ is the residue field of $s$. 
If for each $s\in S$ the fiber 
$G_{\overline{s}}$ is a 
connected reductive group, then 
$G$ is called a {\em reductive $S$-group}. 

\vspace{.5cm}

So, what is a reductive $k$-group 
over an algebraically closed field? 
We take this opportunity 
to define the `unipotent radical' and 
the `radical' of an algebraic group.
Let $G$ be a $k$-group (affine or not). 
There is a maximal 
connected solvable normal
linear algebraic subgroup, 
denoted by 
$\ms{R}(G(\overline{k}))$, and 
it is called the radical of $G$ 
(see~\cite[Lemma 3.1.4]{Brion:CliffordLectures}).
If $\ms{R}(G(\overline{k}))$ is 
trivial, then $G(\overline{k})$ is called semisimple.
The {\em unipotent radical} of $G(\overline{k})$,
denoted by $\ms{R}_{u}(G(\overline{k}))$,
is the maximal connected normal 
unipotent subgroup of 
$G(\overline{k})$. 
If $\ms{R}_{u}(G(\overline{k}))$
is trivial, then $G(\overline{k})$ is called 
reductive. 
Clearly, semisimplicity implies reductivity. 
Notice also that
we have no connectedness 
assumption here. 
If the characteristic of $k$ is 0, then the property of 
reductiveness of the identity component
of $G$ is equivalent to the semisimplicity of 
all linear representations of $G$. 
This equivalence fails in positive characteristics  
(see~\cite[Remark 1.1.13.]{Conrad:Reductive}).

\vspace{.5cm}

In the rest of our paper 
we will focus mainly on 
the ``relative'' reductive group schemes
over fields.
\begin{definition}\label{D:connectedreductive}
Let $k$ be a field. A 
$k$-group $G$ is called 
{\em reductive} if the geometric
fiber $G_{\overline{k}}$
(which we take to be equal to  
$G(\overline{k})$)
is a connected reductive
group in the sense
of the previous paragraph.
\end{definition}

\subsection{Parabolic subgroups.}\label{SS:ParabolicSubgroups}

Let $G$ be a reductive $k$-group.
A subgroup $P (\overline{k})$ of 
$G(\overline{k})$ is called {\em parabolic} 
if $G(\overline{k})/P(\overline{k})$
has the structure of a projective variety. 
More generally, a smooth affine $k$-subgroup $P$ 
of $G$ is called a {\em parabolic subgroup} if 
$P(\overline{k})$ is a parabolic subgroup
of $G(\overline{k})$.

A {\em Borel subgroup} in $G(\overline{k})$ 
is a maximal connected solvable subgroup. 
More generally, 
a parabolic $k$-subgroup $B$ 
of $G$ is called a {\em Borel subgroup} 
if $B(\overline{k})$
is a Borel subgroup in $G(\overline{k})$.

A fundamentally important 
result that is due to  
Borel (see~\cite[Theorem 11.1]{Borel}) 
states that any two Borel subgroups 
are conjugate in $G(k)$, 
and furthermore, 
for any Borel 
$k$-subgroup $B(k)$ 
the quotient $G(\overline{k})/B(\overline{k})$ 
is projective. 
Of course, according to the above 
definition of parabolic $k$-subgroups,
a Borel subgroup in $G(k)$ is a 
(minimal) parabolic $k$-subgroup.

\begin{example}
The upper triangular 
subgroup $\mt{T}_n$ 
of $\mt{GL}_n$ 
is a Borel subgroup. 
As a consequence of 
the Lie-Kolchin Theorem 
(see~\cite[Corollary 10.2]{Borel}) 
any connected solvable 
group admits
a faithful representation 
with image in $\mt{T}_n$. 
\end{example}

It may happen that $G(k)$ 
is defined over a field but 
has no nontrivial Borel subgroup. 
This holds true, 
even for some classical groups, 
as we will demonstrate 
in the next classic example 
(from~\cite{Borel}).

\begin{example}
Let $k$ be a field whose 
characteristic is not 2, and 
let $V$ be $k$-vector space.
Let $Q$ be a nondegenerate 
quadratic form on $V$, and 
let $F$ denote its
symmetric bilinear form.
We assume that $Q$ is 
{\em isotropic}, that is to say,
there exists a nonzero vector 
$v\in V$ such that $Q(v)=0$. 
A subspace 
is called {\em isotropic} if it contains 
an isotropic vector; a subspace
is called {\em anisotropic} if it contains 
no nonzero isotropic vector; a
subspace is called {\em totally 
isotropic} if it consists of isotropic
vectors only. 
A {\em hyperbolic plane}
is a two-dimensional 
subspace $E$ of $V$ with a 
basis $\{e,f\}$ with respect to which 
the restriction of $F$ has the form 
$F(x_1 e + x_2 f, y_1 e + y_2 f ) =
x_1 y_2+x_2y_1$.

By the Witt's Decomposition Theorem
we know that the dimension $q$
of a maximal totally isotropic
subspace is an invariant of 
$Q$. More precisely, it states that 
$V$ contains $q$ linearly
independent hyperbolic planes
$H_1,\dots, H_q$, and $V$ is an  
orthogonal direct sum of the form 
$$
V\cong V_o\oplus \bigoplus_{i=1}^q H_i,
$$
where $V_o$ is an anisotropic 
subspace. Let $Q_o$ denote 
the restriction of $Q$ to $V_o$.

For $i=1,\dots, q$
we choose a basis 
$\{e_i,e_{n-q+i}\}$ for $H_i$ 
in such a way that 
the following identities 
are satisfied:
$$
F(e_i,e_i)=F(e_{n-q+i},e_{n-q+i}) = 0\ \text { and } \
F(e_i, e_{n-q+i})=1.
$$
Here $n$ is the dimension 
of $V$.
Let $\{ e_j :\ j = q+1,\dots, n-q\}$ denote 
a basis for $V_o$. 
For each pair $e_i,e_{n-q+i}$ 
($i=1,\dots,q$) of basis 
elements and $x\in \overline{k}$ 
we have a linear 
map $s_i(x): V(\overline{k}) \rightarrow V(\overline{k})$ 
defined by 
$$
s_i(x) e_i = x e_i,\ \, \, s_i(x) e_{n-q+i} = x^{-1} e_{n-q+i},\ \
\text{ and }\ s_i(x) f = f \ \text{ for }\ f \in H_i^{\perp}.
$$
Clearly, the $s_i(x)$'s are semisimple 
and generate a diagonal torus,
denoted by $T$, 
whose elements expressed 
in the basis $\{e_1,\dots, e_n\}$ of $V$ 
are of the form 
$$
\text{diag} ( a_1,\dots, a_q, 1,\dots, 1,a_1^{-1},\dots, a_q^{-1}), 
\ \ \text{ where } a_i \in \overline{k}.
$$
If $G$ denotes  
$\mathbf{SO}(Q)$, the $k$-group 
consisting of linear 
automorphisms of $V$ that preserves 
the quadratic form $Q$ and of determinant
1, then $S:=T\cap G$ is in fact a 
maximal torus of $G$. 
Note that the centralizer of $S$ in $G$,
denoted by $\ms{Z}_G(S)$, is isomorphic to 
the product $S\times \mathbf{SO}(Q_o)$.

The group $G$ is a reductive $k$-group,
and since $S$ is a maximal torus, 
by definition, $\ms{Z}_G(S)$ 
is a Cartan subgroup of $G$.
Therefore $\ms{Z}_G(S)$ is 
contained in a minimal 
parabolic $k$-subgroup $P$ of $G$.
In fact, $\ms{Z}_G(S)$ is equal 
to the Levi component of $P$. 
Observe that, if $n> 2q+2$, 
then $\mathbf{SO}(Q_o)$
is not commutative, hence 
$\ms{Z}_G(S)$ is not
contained by a Borel subgroup. 
We conclude that 
if $n>2q+2$, then $G$ does not 
have any Borel subgroups.

\end{example}

For connected semisimple 
$k$-groups, where $k$ is a perfect field,
the question of the existence 
of Borel subgroups has a nice
answer.
\begin{theorem}[Ono]
Let $k'$ and $k$ be two 
fields such that
$k'\subset k$ and $k'$ is perfect. 
If $G$ is a connected semisimple 
$k$-group, then $G$ posses
a Borel subgroup $B$ defined 
over $k'$ 
if and only if a maximal torus 
$T$ of $G$ is 
$k'$-split. In this case, 
all Borel subgroups containing 
$T(k')$ are conjugate by the 
elements of 
the group $N_G(T)(k')$, 
where $N_G(T)$
denotes the normalizer 
group of $T$ in $G$.
\end{theorem}

\begin{proof}
See~\cite{Ono}.
\end{proof}

Finally, we finish this section
by mentioning 
another important related result.

\begin{theorem}[Chevalley]
The parabolic subgroups 
of any connected 
linear algebraic $k$-group 
are connected, and moreover, 
the normalizer of a parabolic 
subgroup in $G(k)$ 
is equal to $P(k)$. 
\end{theorem}
\begin{proof}
See~\cite[Section 11]{Borel}.
\end{proof}

\subsection{Actions of group schemes.}\label{SS:Actions}

Let $G$ be a group scheme over $S$. 
An action of $G$ on an $S$-scheme $X$ 
is an $S$-morphism
$$
a : G\times_S X \rightarrow X
$$
such that for any $S$-scheme $T$ the 
morphism $a(T): G(T) \times X(T) \rightarrow X(T)$ 
is an action of $G(T)$ on $X(T)$. 
In this case, $X$ is called a $G$-scheme. 
If the action map $a$ is clear from the context 
we will denote $a(g,x)$ by $g\cdot x$, 
where $g\in G,\ x\in G(S)$.

\vspace{.5cm}

Let $X$ be a $G$-scheme with respect to 
an action $a: G\times X\rightarrow X$.
If $X$ is an affine scheme, 
then for any scheme $S$ 
and $g\in G(S)$ we have an
$\mc{O}_S$-algebra automorphism 
defined as follows:
\begin{align*}
\rho(g):\ \mc{O}_S\otimes_k \Gamma(X,k) 
& \longrightarrow \mc{O}_S\otimes_k \Gamma(X,k) \\
1\otimes f (-) )&\longmapsto 1\otimes ( f \circ a(g^{-1}, -)),
\end{align*}
where $\Gamma(X,k) $ is the global section functor 
applied to $X$.
These automorphisms patch up to give 
a linear representation 
\begin{align}\label{A:linearrep}
\rho: G\rightarrow GL(\Gamma(X,k) ).
\end{align}

\begin{remark}\label{R:Sumihiro}
If $X$ is a $G$-scheme of finite type, 
then the representation (\ref{A:linearrep})
decomposes $\Gamma(X,k)$ 
into a union of finite 
dimensional $G$-submodules 
(\cite[Proposition 2.3.4]{Brion:CliffordLectures}).
Furthermore, if $X$ is affine, 
then there exists a closed 
$G$-equivariant immersion of 
$X$ into a finite dimensional 
representation of $G$ 
(\cite[Proposition 2.3.5]{Brion:CliffordLectures}).
\end{remark}

We close this section by 
mentioning an important
theorem of Brion on 
projective equivariant 
embeddings of algebraic groups.
Let $G$ be a $k$-group,  
and let $H\subseteq G$
be a subgroup scheme. 
An {\em equivariant compactification} of the 
homogenous space $G/H$ is a proper $G$-scheme
$X$ equipped with an open equivariant immersion
$G/H\rightarrow X$ with schematically
dense image. 

\begin{theorem}[Brion]\label{T:projectiveembedding}
Let $G$ be a $k$-group,  
and let $H\subseteq G$
be a subgroup scheme. 
Then $G/H$ has an equivariant
compactification by a 
projective scheme. 
\end{theorem}
\begin{proof}
See~\cite[Section 5.2]{Brion:CliffordLectures}.
\end{proof}

Some remarks are in order:
\begin{enumerate}
\item The homogenous space $G/H$ of 
Theorem~\ref{T:projectiveembedding} 
is quasiprojective; this is 
a well-known theorem for algebraic groups
in the classical sense. But notice here that 
Brion proves the result for not necessarily
reduced algebraic groups;
of course, if the ground
field is of characteristic zero, then Cartier's
theorem implies that $G$ is reduced. 
\item If in addition $G$ is smooth, then
$G/H$ has an equivariant compactification
by a normal projective scheme.
\item If the characteristic of $k$ is 0, then 
every homogenous space has 
a smooth projective equivariant 
compactification.
\item Over any imperfect field,
there exists smooth connected 
algebraic groups having no smooth 
compactification (see~\cite[Remark 5.2.3]{Brion:CliffordLectures}
for an example).
\end{enumerate}

\section{Algebraic monoids}\label{S:AlgebraicMonoids}

For an algebro-geometric introduction 
to the theory of (not necessarily affine) 
monoid and semigroup schemes
we recommend Brion's lecture 
notes~\cite{Brion:FieldsLectures}
(see also~\cite[Chapter II]{DemazureGabriel}
\footnote{This book is one of the few 
if not the only book in 
algebraic geometry that 
acknowledges monoid schemes 
as part of the theory of group schemes.}).

Let us define a monoid scheme
by relaxing the condition of 
invertibility in the definition of group 
schemes. More precisely, 
let $R$ be a commutative ring.
A monoid scheme over $R$ 
is an $R$-scheme
whose functor of points 
factors through the
forgetful functor from the category 
of monoids to the category of sets.

\begin{definition}\label{D:k-monoids}
Let $k$ be an algebraically closed 
field. An {\em algebraic monoid over $k$},
also called a {\em $k$-monoid}, is 
a monoid scheme over $k$ whose
underlying scheme is 
separated and of finite type. 
\end{definition}

\begin{remark}
Our definition of $k$-monoids is 
somewhat more general than the one
that is used by Brion in~\cite{Brion:FieldsLectures}
since we do not assume reducedness 
of the underlying scheme structure. 
\end{remark}

It is not difficult to see that 
the category of $k$-groups
forms a full subcategory
of the category of $k$-monoids. 
We are mainly interested in 
the affine $k$-monoids but let us first 
mention a general result.

Let $M$ be a $k$-monoid, and let 
$G=G(M)$ denote its unit-group. 
Then $M$ is called {\em unit-dense} if $G$ is 
dense in $M$. 
A (weaker) form of the following theorem 
was first proven by Rittatore. 
Here, we will make use of Brion's proof 
from~\cite[Section 3]{Brion:FieldsLectures}.
\begin{theorem}\label{T:RittatoreBrionCan}
Let $M$ be a unit-dense irreducible 
$k$-monoid, and let $G$ denote the unit-group of $M$. 
If $G$ is affine, then so is $M$. 
\end{theorem}

\begin{proof}
If $M$ is reduced, then the result 
follows from~\cite[Theorem 2]{Brion:FieldsLectures}.
If $M$ is not reduced, then we pass to a normalization
$\widetilde{M}$ which factors through the 
reduction $M_{red}$. 
This follows from the fact that $M$ is 
irreducible (hence Noetherian). Therefore, we can apply 
~\cite[\href{http://stacks.math.columbia.edu/tag/0350}{Lemma 28.52.2}]{StacksProject}.
But a normalization is a finite morphism,
therefore $M_{red}\rightarrow M$ is a finite morphism.
Since $M_{red}$ is affine (once again by~\cite[Theorem 2]{Brion:FieldsLectures}), and since 
the image of an affine scheme under a finite morphism
is affine, $M$ is an affine scheme as well.
\end{proof}

\section{From spherical varieties
to spherical spaces}\label{S:FromSpherical}

We start with reviewing
the classification schematics of 
the spherical embeddings over
an algebraically closed field.
Throughout this section $k$
will denote an algebraically
closed field and we assume 
that all varieties are defined 
over $k$. 
As usual, let $G$ be
a reductive group, let $B$ be a Borel subgroup,
and let $T$ be a maximal torus
contained in $B$.
If $K$ is an algebraic 
group, we will denote by
$X^*(K)$ the group of 
characters of $K$. Note that 
$X^*(B)=X^*(T)$. 
This is because the commutator
subgroup of $B$ coincides with  
the unipotent radical $\ms{R}_u(B)$, 
and $B=\ms{R}_u(B) \rtimes T$.

Let $Y$ denote $G/H$,
where $H$ is a spherical
subgroup of $G$.
Quotients of affine 
groups, in particular $Y$, 
have the structure of 
a quasiprojective variety. 
Recall that $Y$ is a spherical 
$G$-variety if and only if
$B$ has only finitely many
orbits with respect to the 
left multiplication action on $Y$. 
By a theorem of Brion,
this is equivalent to the 
statement that $B$ has 
an open orbit in $Y$. 
Thus, it should come 
as no surprise 
that the $B$-invariant 
rational functions  
on $Y$ are among 
the main players in this game.

The space of $B$-semiinvariant
rational functions on $Y$ is denoted by 
$k(Y)^{(B)}$. In other words, 
\begin{align}\label{A:semiinvariants}
k(Y)^{(B)}=\{ f\in k(Y):\ 
b\cdot f = \chi(b) f
\ \text{ for all $b\in B$ and 
for some character $\chi$ of $B$} \}.
\end{align}
The gist of the classification
schematics for spherical 
varieties will take place 
inside the vector space 
$\textbf{Hom}_\Z ( X^*(B),\Q)$. 
Indeed, it is easy to check  
that $k(Y)^{(B)}$ is a subgroup
of $k(Y)$ and that the assignment
\begin{align}\label{A:assignment}
k(Y)^{(B)} &\longrightarrow X^*(B)\\
f &\rightsquigarrow \chi_f, \notag
\end{align} 
where $\chi_f$ is a 
character as in (\ref{A:semiinvariants}) 
is an injective  
group homomorphism.
We will denote the image 
of (\ref{A:assignment}) by 
$\varOmega_Y$, and 
we will denote the $\Q$-vector 
space associated
with the dual of $\varOmega_Y$
by $\mc{Q}_Y$. 
In other words, 
\begin{itemize}
\item $\varOmega_Y :=\{ \chi \in X^*(B) :\ b\cdot f = \chi(b) f
\ \text{ for all $b\in B$ and for some $f\in k(Y)^{(B)}$} \}$; 
\item $\mc{Q}_Y :=  \textbf{Hom}_\Z ( \varOmega_Y, \Q)$.
\end{itemize}
We occasionally refer to $\varOmega_Y$ 
as the character group of the homogenous
variety $Y$ since in the special case, where
$Y=G\times G /\text{diag} (G) \cong G$
viewed as a spherical $G\times G$-variety, 
the character group of $Y$ 
is isomorphic to the 
`ordinary' character group,
$\varOmega_{Y} 
\cong X^*(B)$.

We look closely  
at the divisors 
and their invariants on 
$Y$. 
A function $\nu: k(Y)\rightarrow \Q$ 
is called a $\Q$-valued 
discrete valuation on $k(Y)$ if 
for every $a,b\in k(Y)^*$, 
we have:
\begin{enumerate}
\item $\nu(k(Y)^*) \cong \Z$ \  and  \ $\nu(k)=\{0\}$; 
\item $\nu(ab) = \nu(a)+\nu(b)$;
\item $\nu(a+b) \geq \min \{ \nu(a),\nu(b)\}$ provided $a+b\neq 0$.
\end{enumerate}
We notice here that 
every $\Q$-valued 
discrete valuation
$\nu$ on $k(Y)$ defines 
a function on $\varOmega_Y$.
More precisely,
there is a map
\begin{align*}
\rho:\ \{\text{$\Q$-valued 
discrete valuations on $k(Y)$}\} 
&\longrightarrow \mc{Q}_Y \\
\nu 
&\longmapsto \rho_\nu 
\end{align*}
such that $\rho_\nu$ is the 
map that sends 
an element $\chi=\chi_f$ 
of $\varOmega_Y$
to $\nu(f)$, where $f$ is 
the $B$-semiinvariant 
that specifies $\chi$ 
(so we wrote $\chi=\chi_f$).
Indeed, since $\chi$ is defined 
uniquely by $f$ (up to 
a scalar), 
$\rho_\nu$ is well-defined.
Moreover, it is 
easy to check that 
$\rho_\nu( \chi_f \chi_g) = \rho_\nu(\chi_f) + \rho_\nu (\chi_g)$,
hence $\rho_\nu\in \mc{Q}_Y$.

\begin{itemize}
\item 
A valuation 
$\nu$ is called $G$-invariant 
if $\nu(g\cdot a) = \nu(a)$ for every 
$g\in G$ and $a\in k(Y)$. 
We will denote  the set of all 
$\Q$-valued $G$-invariant 
discrete valuations on $Y$
by $\mc{V}_Y$.
\end{itemize}
It is not completely 
obvious, but nevertheless
true, that the restriction 
$$
\rho \vert_{\mc{V}_Y} :
\mc{V}_Y \longrightarrow \mc{Q}_Y
$$
is an injective homomorphism 
of abelian groups.

\vspace{.5cm}

The reason for which we 
started looking at 
discrete valuations 
defined on $Y$ in 
the first place is 
because many important 
discrete valuations come 
from irreducible 
hypersurfaces. 
From now on, 
we will refer to
irreducible hypersurfaces
as prime divisors.
For any spherical action 
of $G$ on $Y$, we will 
consider the set of all 
$B$-stable prime
divisors in $Y$.
More generally, if $Y'$ is 
a normal spherical
$G$-variety, then 
{\em a color of $Y'$} 
is defined as a $B$-stable,
but not $G$-stable, prime divisor. 
In our case, 
since $G$ acts transitively 
on $Y=G/H$, any $B$-stable 
prime divisor in $Y$ 
is non-$G$-stable.
\begin{itemize}
\item The set of 
all $B$-stable, but 
not $G$-stable, 
prime divisors in $Y$ 
is denoted by 
$\mc{D}_{Y}$.
The elements of 
$\mc{D}_{Y}$
are called the colors 
of $Y$. 
\end{itemize}
Let us point out that 
on a noetherian 
integral separated scheme, 
which is regular in 
codimension one, 
the local ring associated 
with each 
prime divisor 
is a discrete 
valuation ring (DVR).
In particular, 
since our $Y$ 
is a smooth variety, 
the local ring of a color 
of $Y$ at its generic point 
is a DVR, and we have
the map 
\begin{align}\label{A:rho}
\widetilde{\rho} 
: \mc{D}_{Y} \rightarrow \mc{Q}_Y,
\end{align}
which is defined 
as the composition of 
$\rho$ with the map
that assigns a color  
to the corresponding
discrete valuation.

\begin{definition}
The colored cone of $Y=G/H$
is the pair $(\mc{C}_Y,\mc{D}_{Y})$
where $\mc{C}_Y$ is the cone in 
$\mc{Q}_Y$ that is generated
by $\rho(\mc{V}_Y)$ and 
$\widetilde{\rho}(\mc{D}_{Y})$.
\end{definition}

So far what we
have are 
some `birational invariants'
that are defined solely 
for $Y=G/H$, and 
we have not given 
any indication of
how they are related
to its embeddings. 
To see how
all these basic ingredients 
come together to play 
a role, next, we 
introduce the notion 
of a colored fan.
This will give us a generalization  
of the combinatorial 
classification of toric varieties. 
\vspace{.5cm}

Let $\overline{Y}$ 
be a $G$-equivariant 
embedding of $Y$. 
Let 
$\mc{D}_{\overline{Y}} $
denote the set of 
$B$-stable, but not
$G$-stable, 
prime divisors of $\overline{Y}$. 
Clearly, 
$\mc{D}_{Y}$ 
is a subset of 
$\mc{D}_{\overline{Y}}$.
Since $Y$ is the 
open $G$-orbit in 
$\overline{Y}$, we have 
$k(Y)= k(\overline{Y})$. 
In particular, 
there is an extension 
of (\ref{A:rho}) to a
map 
$\overline{\rho}
: \mc{D}_{\overline{Y}}
\rightarrow \mc{Q}_Y$.
Let 
$$
\pi : \mc{D}_{\overline{Y}} \rightarrow 
\mc{D}_{Y}
$$
denote the partially defined map 
$\pi(S)=S\cap Y$,
whenever $S\cap Y$ is an element 
of $\mc{D}_{Y}$.
Let us mention in passing 
that both of the sets 
$\mc{D}_{Y}$
and $\mc{D}_{\overline{Y}}$
contain a finite number 
of elements since
$\overline{Y}$ is spherical. 
Note also that 
the set of $G$-invariant 
discrete valuations 
on $\overline{Y}$ is equal
to $\mc{V}_Y$.

\begin{table}[htp]
\begin{tabular}{lll}
$Y'$ &: & a $G$-orbit in $\overline{Y}$\\
$\nu_S$ &: & the discrete 
valuation in $k(Y)$ associated with 
a prime divisor $S$ of $\overline{Y}$\\
$\mc{V}_{\small{Y'\hookrightarrow \overline{Y}}}$ &: & the set 
of $G$-invariant valuations in $k(Y)$ of the form $\nu_S$ 
with $Y'\subset S \subset \overline{Y}$\\
$\mc{D}_{\small{Y'\hookrightarrow \overline{Y}}}$ &: & the set 
of colors $D\in \mc{D}_{Y}$ such that 
$\exists S \in \mc{D}_{\overline{Y}}$ with  
$Y'\subset S$ and 
$D=\pi(S)$ \\
$\mc{C}_{\small{Y'\hookrightarrow \overline{Y}}}$ &: & the 
cone in $\mc{Q}_Y$ that is generated by the images 
$\rho(\mc{V}_{\small{Y'\hookrightarrow \overline{Y}}})$
and $\widetilde{\rho}(\mc{D}_{\small{Y'\hookrightarrow \overline{Y}}})$
\end{tabular}
\end{table}

\begin{itemize}
\item The pair $(\mc{C}_{\small{Y'\hookrightarrow \overline{Y}}},
\mc{D}_{\small{Y'\hookrightarrow \overline{Y}}})$ is called 
the {\em colored cone} of the $G$-orbit $Y'$.
A {\em face} of $(\mc{C}_{\small{Y'\hookrightarrow \overline{Y}}},
\mc{D}_{\small{Y'\hookrightarrow \overline{Y}}})$
is a pair of the form $(\mc{C},\mc{D})$, 
where $\mc{C}$ is a face of 
$\mc{C}_{\small{Y'\hookrightarrow \overline{Y}}}$
such that 
\begin{enumerate}
\item[(i)] $\mc{C} \cap 
(\mc{C}_{\small{Y'\hookrightarrow \overline{Y}}})^0 \neq \emptyset$;
\item[(ii)] $\widetilde{\rho}^{-1}(\mc{C})\cap
\mc{D}_{\small{Y'\hookrightarrow \overline{Y}}} \neq \emptyset$.
\end{enumerate}
\end{itemize}
Any colored cone satisfies the following defining properties:
\begin{enumerate}
\item[\textbf{C1.}] The cone
$\mc{C}_{\small{Y'\hookrightarrow \overline{Y}}}$ 
is generated by 
$\widetilde{\rho}(\mc{D}_{\small{Y'\hookrightarrow \overline{Y}}})$
and a finite number of 
vectors of the form 
$\rho(\nu_S)$, 
where
$S\in \mc{V}_{\small{Y'\hookrightarrow \overline{Y}}}$.
\item[\textbf{C2.}] The relative interior of 
$\mc{C}_{\small{Y'\hookrightarrow \overline{Y}}}$ 
has a nonempty intersection with the set 
$\rho(\mc{V}_{\small{Y'\hookrightarrow \overline{Y}}})$.
\item[\textbf{C3.}] $\mc{C}_{\small{Y'\hookrightarrow \overline{Y}}}$
is strictly-convex, that is to say,
$\mc{C}_{\small{Y'\hookrightarrow \overline{Y}}} \cap
(-\mc{C}_{\small{Y'\hookrightarrow \overline{Y}}})=\{0\}$.
\item[\textbf{C4.}] $0$ is not an element of 
$\widetilde{\rho}(\mc{D}_{\small{Y'\hookrightarrow \overline{Y}}})$.
\end{enumerate}

\vspace{.5cm}

Now we are ready 
to introduce the 
combinatorial
objects which 
parametrize the
$G$-equivariant embeddings
of $Y$.
\begin{definition}
The following (finite) set is called the 
{\em colored fan} of $\overline{Y}$:
\begin{align*}
\mc{F}_{\overline{Y}} :=
\{ 
(\mc{C}_{\small{Y'\hookrightarrow \overline{Y}}},
\mc{D}_{\small{Y'\hookrightarrow \overline{Y}}}) :\
Y' \text{ is a $G$-orbit in $\overline{Y}$} \}.
\end{align*}
\end{definition}

\vspace{.5cm}
The colored fans satisfy the following 
defining properties:
\begin{enumerate}
\item[\textbf{F1.}] Every face of 
a colored cone in $\mc{F}_{\overline{Y}}$ 
is an element of $\mc{F}_{\overline{Y}}$. 
\item[\textbf{F2.}] For every $G$-invariant
valuation $\nu$ in $\mc{V}_Y$, 
there exists at most one colored 
cone $(\mc{C},\mc{D})$ in
$\mc{F}_{\overline{Y}}$ such that 
$v\in \mc{C}^0$.
\end{enumerate}

It is easy to 
make abstract
versions of colored fans.
Let $V$ be a 
finite dimensional
vector space over $\Q$. 
Starting with a 
subset $\mc{V}$ of $V$
and a finite set $\mc{D}$ 
together with a set map
$\widetilde{\rho}:\mc{D}
\rightarrow V$, 
we define a {\em colored fan
associated with 
$(V,\mc{V},\mc{D},\widetilde{\rho})$}
as a finite collection of 
pairs $(\mc{C},\mc{E})$,
where 
$\mc{C}$ is a 
cone in $V$
and $\mc{E}$ is a 
subset of $\mc{D}$ 
satisfying the properties
\textbf{F1},\textbf{F2},
\textbf{C1}$-$\textbf{C4}.
Of course, $\mc{V}$ 
plays the role of $\rho(\mc{V}_Y)$
in $\mc{Q}_Y$ 
and $\widetilde{\rho}: \mc{D}\rightarrow
V$ plays
the role of 
$\widetilde{\rho}: \mc{D}_{Y}\rightarrow 
\mc{Q}_Y$.

Let $H$ and $H'$
be two closed subgroups
in $G$ such that the 
homogenous varieties
$$
Y:=G/H \ \text{ and } Z:=G/H'
$$
are spherical. 
Let 
$$\varphi: Y\rightarrow Z$$ 
be a morphism of 
varieties. If $\varphi$
is $G$-equivariant, then the 
resulting map on 
the character groups 
$\varphi^*: X^*(Z)\rightarrow X^*(Y)$
is injective, hence, the `dual' 
linear map
$\varphi_*: \mc{Q}_{Y}\rightarrow \mc{Q}_Z$
is surjective. 
Furthermore, we have 
$\varphi_*(\mc{V}_Y)=\mc{V}_Z$.
Let $\mc{D}_\varphi^c$ denote 
the set of colors of $Y$ that 
are mapped into $Z$ dominantly,
$$
\mc{D}_\varphi^c := \{ D\in \mc{D}_Y:\
\text{ $\varphi(D)$ is dense in $Z$} \}.
$$
In other words, $\mc{D}_\varphi^c$ 
is the set of colors 
of $Y$ which are too big,
so, we may ignore(!) them
in the combinatorial setup. 
We set 
$$
\mc{D}_\varphi := \mc{D}_{Y} - \mc{D}^c_\varphi.
$$

\begin{definition}
Let $\varphi : Y\rightarrow Z$ 
be a $G$-equivariant
morphism between two spherical 
homogenous $G$-varieties. 
Let $\overline{Y}$ 
and $\overline{Z}$ denote
two equivariant 
embeddings of $Y$ and $Z$,
respectively. Let $Y'$ and $Z'$
be two $G$-orbits in $\overline{Y}$
and $\overline{Z}$, respectively. 
The map 
$\varphi$ is said to be 
a {\em morphism} between 
the colored cones 
$(\mc{C}_{\small{Y'\hookrightarrow \overline{Y}}},
\mc{D}_{\small{Y'\hookrightarrow \overline{Y}}})$
and
$(\mc{C}_{\small{Z'\hookrightarrow \overline{Z}}},
\mc{D}_{\small{Z'\hookrightarrow \overline{Z}}})$ 
if we have 
\begin{enumerate}
\item $\varphi_*(\mc{C}_{\small{Y'\hookrightarrow \overline{Y}}})
\subseteq \mc{C}_{\small{Z'\hookrightarrow \overline{Z}}}$, and 
\item $\varphi(\mc{D}_\varphi \cap\mc{D}_{\small{Y'\hookrightarrow \overline{Y}}}) \subseteq
\mc{D}_{\small{Z'\hookrightarrow \overline{Z}}}$.
\end{enumerate}
The map $\varphi$ is said to 
be a {\em morphism} between
the colored fans $\mc{F}_{\overline{Y}}$ and $\mc{F}_{\overline{Z}}$ if for every cone $(\mc{C},\mc{D})$
in $\mc{F}_{\overline{Y}}$ there exists a cone 
$(\mc{C}',\mc{D}')$ in $\mc{F}_{\overline{Z}}$
such that 
$\varphi : (\mc{C},\mc{D})\rightarrow 
(\mc{C}',\mc{D}')$ is a morphism of cones. 
\end{definition}

The following result,
which is proven by
Knop in~\cite{Knop89}, 
is a  generalization
of the classification 
result of 
Luna and Vust for
simple embeddings.

\begin{theorem}[Knop]
Let $Y$ be a spherical
homogenous $G$-variety,
and let $B$ be a Borel 
subgroup in $G$. 
The assignment 
$\overline{Y} \rightsquigarrow 
\mc{F}_{\overline{Y}}$ 
is a bijective 
correspondence between 
the isomorphism classes 
of $G$-equivariant
embeddings of $Y$ and 
the isomorphism
classes of 
colored fans associated with 
$(\mc{Q}_Y,\mc{V}_Y,\mc{D}_{Y},\widetilde{\rho})$.
In fact, this assignment
is an equivalence 
between the category
of equivariant embeddings of 
$Y$ and the category 
of colored fans
associated with 
$(\mc{Q}_Y,\mc{V}_Y,\mc{D}_Y,\widetilde{\rho})$.
\end{theorem}

\begin{remark}
As we mentioned 
before the theorem
of Knop, the role 
of colored fans for 
simple embeddings
was already known.
In fact, Luna and Vust
had shown in~\cite{LunaVust} 
that the colored cone 
$(\mc{C}_{\small{Z\hookrightarrow \overline{Y}}},
\mc{D}_{\small{Z\hookrightarrow \overline{Y}}})$, 
where $Z\hookrightarrow \overline{Y}$
is the closed orbit of
$\overline{Y}$, 
uniquely determines 
$\overline{Y}$. 
\end{remark}

\begin{remark}
It is not difficult to check that 
all definitions pertaining to 
the colored cones make sense 
(definable) if we use a separably
closed field instead of an algebraically
closed field. 
\end{remark}

\subsection{Quasiprojective
colored fans.}\label{SS:Quasiprojective}

It is useful to know when 
an equivariant embedding
of a spherical homogenous
variety is affine, projective,
or more generally quasiprojective.
Such criteria are found 
by Brion in~\cite{Brion:Duke}. Here
we only give Brion's criterion for 
quasiprojectiveness.

\begin{theorem}[Brion]\label{T:quasiprojective}
Let $\mc{F}_{\overline{Y}}$ 
be the colored fan of 
an equivariant embedding 
$\overline{Y}$ of a spherical
homogenous $G$-variety $Y$. 
In this case, $\overline{Y}$ 
is quasiprojective if and only if 
for each colored cone 
$C_Z:=(\mc{C}_{\small{Z\hookrightarrow \overline{Y}}},
\mc{D}_{\small{Z\hookrightarrow \overline{Y}}})$
in $\mc{F}_{\overline{Y}}$ 
there exists a linear form,
denoted by $\ell_Z$, 
on $\mc{Q}_{\overline{Y}}$
such that the following two 
conditions are
satisfied: 
\begin{enumerate}
\item If $C_Z=
(\mc{C}_{\small{Z\hookrightarrow \overline{Y}}},
\mc{D}_{\small{Z\hookrightarrow \overline{Y}}})$
and $C_{Z'}=
(\mc{C}_{\small{Z'\hookrightarrow \overline{Y}}},
\mc{D}_{\small{Z'\hookrightarrow \overline{Y}}})$ 
are two elements 
from $\mc{F}_{\overline{Y}}$, then 
the restrictions of the corresponding
linear forms onto 
$\mc{C}_{\small{Z\hookrightarrow \overline{Y}}}
\cap \mc{C}_{\small{Z'\hookrightarrow \overline{Y}}}$
are the same.
\item If $C_Z=
(\mc{C}_{\small{Z\hookrightarrow \overline{Y}}},
\mc{D}_{\small{Z\hookrightarrow \overline{Y}}})$
and $C_{Z'}=
(\mc{C}_{\small{Z'\hookrightarrow \overline{Y}}},
\mc{D}_{\small{Z'\hookrightarrow \overline{Y}}})$ 
are two distinct elements 
from $\mc{F}_{\overline{Y}}$, and if a vector
$\chi \in \mc{Q}_{\overline{Y}}$ lies 
in the intersection of the interior of 
$\mc{C}_{\small{Z\hookrightarrow \overline{Y}}}$ 
with the image $\rho(\mc{V}_{\overline{Y}})$,
then $\ell_Z (\chi) > \ell_{Z'}(\chi)$. 
\end{enumerate}
\end{theorem}

Following Huruguen,
we call a colored fan
whose cones satisfy the 
requirements of 
Theorem~\ref{T:quasiprojective}
a {\em quasiprojective colored fan}. 
We know from (the remarks following) 
Theorem~\ref{T:projectiveembedding}
that there are plenty of quasiprojective colored 
fans, especially over perfect fields.

With this precise definition of 
quasiprojectiveness at hand, now 
we are able to state Huruguen's 
result.

\begin{theorem}[Huruguen]\label{T:Huruguenmain}
Let $k$ be a perfect field, let $G$ be a 
connected reductive group that is 
defined over $k$, and 
let $\overline{Y}(\overline{k})$ be an embedding 
of a spherical homogenous 
spherical $G$-variety $Y$ defined over $\overline{k}$. 
We assume that the fan of $\overline{Y}$ 
is $\varGamma$-stable. In this case, 
$\overline{Y}$ admits a $k$-form if and only if 
for every cone 
$C_Z:=(\mc{C}_{\small{Z\hookrightarrow \overline{Y}}},
\mc{D}_{\small{Z\hookrightarrow \overline{Y}}})$
in $\mc{F}_{\overline{Y}}$,
the colored fan consisting of the cones 
$
(\sigma(\mc{C}_{\small{Z\hookrightarrow \overline{Y}}}),
\sigma(\mc{D}_{\small{Z\hookrightarrow \overline{Y}}}))
$, $\sigma \in \varGamma$
as well as all of its faces are quasiprojective.
\end{theorem}
\begin{proof}
See~\cite[Theorem 2.26]{Huruguen}.
\end{proof}

\subsection{Spherical spaces over arbitrary 
fields.}\label{SS:WedhornsClassification}

We will start with 
giving a brief 
summary of 
Wedhorn's work on 
the classification of
spherical spaces.
For all unjustified
claims (and for some definitions)
we refer the reader to
~\cite{Wedhorn} and 
to the references therein.

\begin{definition}
Let $k$ be a field, and let
$G$ be a reductive $k$-group.
Recall that this amounts to 
the requirement that
$G_{\overline{k}}$ is a 
connected reductive group. 
According to~\cite[Remark 2.2]{Wedhorn},
an algebraic space 
$X$ over $k$ with an
action of $G$ 
is $G$-spherical 
if $X_{\overline{k}}$ is 
a spherical 
$G_{\overline{k}}$-variety. 
\end{definition}

Let $\overline{k}$ denote
a fixed algebraic closure of $k$, let $k_s$ denote
the separable closure of $k$, and
let us denote 
by $\varGamma$
the Galois group of 
the extension $k_s/k$. 
(Here, we are intentionally
vague about our choices
because it does not matter
which separable closure 
we choose.)
In the sequel we will
look at continuous 
and linear actions 
of $\varGamma$ 
on some structures. 
When we speak of 
a continuous action of $\varGamma$
on a set $X$, we will 
treat $X$ with the 
discrete topology.
The important point 
here is that if $X$ is a 
finite set, or,
if the action of $\varGamma$
is linear on some 
finite dimensional vector 
space $X$, then 
the action is continuous
if and only if it factors through 
some finite discrete quotient 
of $\varGamma$. 
This fact should alleviate 
a possible pain of confronting 
a large absolute Galois 
group such as $\varGamma$ 
of $\Q_s/\Q$.

If $\overline{Y}$ is a 
spherical $G$-space,
then there 
exists a homogenous 
spherical 
$G$-space $Y$ 
such that 
$\overline{Y}$ is a spherical
embedding of $Y$. 
This actually 
amounts to the
statement that
$Y$ is the unique
{\rm open} minimal $G$-invariant
subspace of $\overline{Y}$.
By definition, a $G$-invariant
subspace in
an algebraic space $\overline{Y}$ is 
{\em minimal}
if there exists no proper 
non-empty $G$-invariant 
subspace of $\overline{Y}$.

\begin{theorem}[Wedhorn]
\label{T:sepclosed}
Let $G$ be a reductive 
$k$-group, and let $\overline{Y}$ be a
spherical $G$-scheme 
viewed as an equivariant
embedding of the 
spherical homogenous
scheme $G/H$. 
If $k$ is separably closed, 
then the assignment 
$
\overline{Y} \rightsquigarrow \overline{Y}_{\overline{k}}
$
induces a bijection 
between 
the isomorphism classes 
of spherical embeddings
of $G/H$ over $k$ and 
the isomorphism classes 
of spherical
embeddings of 
$G/H$ over $\overline{k}$. 
\end{theorem}

Notice that 
the bijection between
isomorphism classes
that is mentioned in 
Theorem~\ref{T:sepclosed} 
is essentially 
the application of the 
base change functor
from $k$ to $\overline{k}$. 
In general, 
this does not 
give an equivalence 
of categories.
A straightforward example is 
produced by the left
translation action of 
$G=\mathbf{G}_m$
on $Y=\mathbf{G}_m$.
Luckily, since the 
definition of colored
fans works over 
separably closed fields,
and since 
we have faithfully flat descent
upon restriction, 
the classification 
reduces to the 
classification 
over algebraically
closed fields.
The caveat is that 
one needs to consider 
all $G$-invariant
minimal subschemes
of the spherical space.

\begin{corollary}
Let $G$ be a reductive 
$k$-group, and let $\overline{Y}$ be a
spherical $G$-space 
viewed as an equivariant
embedding of the 
spherical homogenous
space $Y:=G/H$. 
If $k$ is separably closed, 
then the assignment 
$\overline{Y} \rightsquigarrow 
(\mc{C}_{\small{Y'\hookrightarrow \overline{Y}}},
\mc{D}_{\small{Y'\hookrightarrow \overline{Y}}})_{Y'}$,
where $Y'$ runs over all
minimal $G$-invariant 
subschemes of $\overline{Y}$,  
is an equivalence 
between the category
of equivariant embeddings of 
$Y$ (over $k$) and the category 
of colored fans
associated with 
$(\mc{Q}_Y,\mc{V}_Y,\mc{D}_Y,\widetilde{\rho})$.
\end{corollary}

Of course, the theorem
and its corollary
that we just presented
here give us something 
new (compared to Knop's 
theorem) 
only when the characteristic
of $k$ is nonzero. 
Now we proceed 
with the general case and 
assume that 
$G$ is a reductive
$k$-group.
Let $Y$ be 
a spherical homogenous
$G$-variety, and let
$\overline{Y}$ be a 
spherical embedding of $Y$.
Both of $Y$ and
$\overline{Y}$ are assumed 
to be defined
over $k$.
Note that Borel subgroups
always exist over 
separably closed fields, 
whence we fix 
a Borel subgroup $B$
in $G$ despite 
the fact that $B$ may not
have any $k$-rational points.
This is where we
start to notice a departure 
from Huruguen's work.

There is a natural 
action of the Galois group 
$\varGamma$ on 
the space of $B$-semiinvariants
$k_s(Y)^{(B)}$.
In particular, $\varGamma$ 
acts on the $k_s$-vector space 
$\varOmega_{Y_{k_s}}$
continuously and linearly. 
Moreover, it acts continuously
on the valuation cone $\mc{V}_{Y_{k_s}}$
as well as on the set of colors 
$\mc{D}_{Y_{k_s}}$, and the 
maps $\rho: \mc{V}_{Y_{k_s}}
\rightarrow \mc{Q}_{Y_{k_s}}$
and 
$\widetilde{\rho}: \mc{D}_{Y_{k_s}}
\rightarrow \mc{Q}_{Y_{k_s}}$
are $\varGamma$-equivariant.

\begin{itemize}
\item A colored fan 
$\mc{F}_{\overline{Y}_{k_s}}$
is said to be {\em $\varGamma$-invariant} 
if its colored cones are permuted by
the action of $\varGamma$.
\end{itemize}

\begin{theorem}[Wedhorn]
Let $G$ be a reductive 
$k$-group, and let $\overline{Y}$ be a
spherical $G$-space 
viewed as an equivariant
embedding  
of the spherical homogenous
$G$-space $Y:=G/H$, 
which is defined
over $k$. 
Then the assignment 
$\overline{Y}_{k_s}  \rightsquigarrow 
(\mc{C}_{\small{Y'\hookrightarrow \overline{Y}_{k_s}}},
\mc{D}_{\small{Y'\hookrightarrow \overline{Y}_{k_s}}})_{Y'}$,
where $Y'$ runs over all
minimal $G$-invariant 
subschemes of $\overline{Y}_{k_s}$, 
induces an equivalence 
between the category
of equivariant embeddings of $Y$ 
over $k$ and the category 
of $\varGamma$-invariant 
colored fans associated with 
$(\mc{Q}_{Y_{k_s}},\mc{V}_{Y_{k_s}},
\mc{D}_{Y_{k_s}},\widetilde{\rho})$.
\end{theorem}

\section{Reductive monoids 
over arbitrary fields}\label{S:ReductiveMonoids}

In this section we will consider
the reductive monoids that are 
defined over arbitrary fields. 
We will show how Wedhorn's 
theorems are applicable to
the algebraic monoid setting.

\begin{definition}\label{D:reductivemonoid}
A {\em reductive} $k$-monoid
is a $k$-monoid whose unit-group
is a reductive $k$-group in the 
sense of Definition~\ref{D:connectedreductive}.
\end{definition}

In particular, according to 
our Definition~\ref{D:reductivemonoid},
the unit-group of a reductive $\overline{k}$-monoid
is a connected reductive monoid,
conforming with our tacit assumption
from the introduction as well
as with that of~\cite{Rittatore98}.

\begin{remark}
Let $M$ be a reductive monoid 
defined over an algebraically
closed field, and let $G$ denote its 
unit-group. The following results are 
recorded in~\cite{Rittatore98}:
\begin{enumerate}
\item $G$
is dense in $M$; 
\item $M$ is 
affine;
\item the reductive monoids
are exactly the affine 
$G\times G$-embeddings
of reductive groups;
\item the commutative reductive
monoids are exactly the affine
embeddings of tori;
\item the isomorphism classes
of reductive monoids 
with unit-group $G$ are
in bijection with the strictly 
convex polyhedral
cones of $\mc{Q}_G$
generated by all of the colors 
and a finite set of elements
from $\mc{V}_G$. 
\end{enumerate}
\end{remark}

\vspace{.5cm}

Now we propose 
a definition for `monoid algebraic spaces.' 
Probably this definition 
exists in the literature, 
however, we could not locate it. 
For our monoid space definition, 
once again, we will 
relax the definition of 
a group algebraic space 
(as given in Stacks Project Tag 043G).

\begin{definition}
Let $S$ be a scheme, and 
let $B$ be an algebraic space
that is separated over $S$.
\begin{itemize}
\item A {\em monoid algebraic space 
over $B$} is a pair $(M,m)$, 
where $M$ is a separated algebraic 
space over $B$ and 
$m: M\times_B M \rightarrow M$
is a morphism of algebraic spaces 
over $B$ with the property
that, for every scheme $T$ 
over $B$, the pair $(M(T),m)$ is 
a monoid.  
\item A {\em morphism} $\psi : (M,m)\rightarrow (M',m')$
of monoid algebraic spaces over $B$ 
is a morphism $\psi: M\rightarrow M'$
of algebraic spaces such that, for every $T/B$,
the induced map $\psi: M(T)\rightarrow M'(T)$ 
is a homomorphism of monoids.
\end{itemize}
\end{definition} 

\begin{definition}
A {\em reductive monoid space} over a scheme $S$ 
is a monoid algebraic space $M$ over $S$ such that 
$M\rightarrow S$ is flat, of finite presentation over $S$, 
and for all $s\in S$ the geometric fiber $M_{\overline{s}}$ 
is a reductive $\overline{k}$-monoid.  
\end{definition}

Clearly, if a monoid algebraic space $M$ 
over a field $k$ is a scheme,
then $M$ is a $k$-monoid 
in the sense of Definition~\ref{D:k-monoids}
but the converse is not true. 
Indeed, in~\cite{Huruguen} Huruguen  
has found an example of a
smooth toric variety of dimension 3 
that is split over a quadratic 
extension of $k$, having  
no $k$-forms. This pathological
example shows that even 
for the purposes of classifying  
reductive monoids over an arbitrary
field one needs to venture into the category of algebraic spaces.

Extending Rittatore's classification
to reductive monoid spaces,
we record the following observations
which are simple corollaries
of Wedhorn's theorems
combined with Rittatore's results. 

\vspace{.5cm}

Recall that the 
reductive monoids 
with unit-group $G$
are $G\times G$-equivariant
embeddings of $G$.
When we speak of `colors'
in this context,
we always mean the colors
of $G$ as a $G\times G$-spherical
$k$-group.

\begin{theorem}\label{T:Extension to monoids}
Let $k$ be a field, and 
let $G$ denote a 
reductive $k$-group. 
Let $M$ be a reductive 
monoid space 
with $G$ as the group of invertible 
elements.

\begin{enumerate}
\item 
If $k$ is separably closed, then 
the assignment 
$M\rightsquigarrow 
(\mc{C}_{Y'\hookrightarrow M},\mc{D}_{G})_{Y'}$,
where $Y'$ runs over all 
minimal $G\times G$-invariant 
subschemes of $M$, is an 
equivalence between 
the category of reductive monoid 
spaces over $k$ and the 
category of strictly 
convex colored 
polyhedral cones of $\mc{Q}_G$
generated by all of the colors of $G$
and a finite set of elements from $\mc{V}_G$. 

\item 
Let $k_s$ be a separable closure of $k$. 
If $k$ is properly contained 
in $k_s$, then the following categories are equivalent:
\begin{enumerate}
\item 
the category of reductive monoid 
spaces over $k$ with unit-group $G$;
\item the 
category of 
$\varGamma$-invariant 
strictly 
convex colored 
polyhedral cones of $\mc{Q}_{G_{k_s}}$
generated by all of the colors of $G_{k_s}$
and a finite set of elements of $\mc{V}_{G_{k_s}}$. 
Here, $\varGamma$ is the 
Galois group of 
the extension $k_s/k$. 
\end{enumerate}
\end{enumerate}
\end{theorem}

It is now desirable to know 
exactly which reductive monoid
schemes over a field have 
a $k$-form.

\begin{theorem}
Let $k$ be a perfect field, 
let $M$ be a reductive monoid 
defined over $\overline{k}$ 
with unit-group $G$, and 
assume that $G$ is defined over $k$. 
In this case, $M$ has a $k$-form
if and only if its colored
fan, which is a strictly convex
polyhedral cone, is stable under
the action of absolute Galois
group of $k\subset \overline{k}$. 
\end{theorem}

\begin{proof}
By Theorem~\ref{T:RittatoreBrionCan},
we know that $M$ is affine, therefore,
its colored fan is automatically quasiprojective. 
Now our result follows from Theorem~\ref{T:Huruguenmain}.
\end{proof}

\subsection{$k$-forms of lined closures.}\label{SS:LinedClosure}

For the next application we 
restrict our attention to the
field of complex numbers, 
and we assume that the reader 
is familiar with the highest 
weight theory.

It is well known that any complex irreducible
affine monoid $M$ admits a faithful 
finite dimensional rational monoid representation. 
In other words, there exists a finite dimensional
vector space $V$ and an injective  
monoid homomorphism
$$
\rho: M \rightarrow End(V),
$$
which is a morphism of varieties 
(see~\cite{Putcha}). 
We notice, in the light of Remark~\ref{R:Sumihiro},
that this fact holds true more generally 
for all irreducible affine $k$-monoids,
where $k$ is an arbitrary field. 
In particular, 
$(V,\rho\vert_G)$ is a faithful rational 
representation of the unit-group $G$,
and $M\cong \overline{\rho(G)}$ in 
$End(V)$. 
In this case, we will write 
\begin{align}\label{A:notation}
M= M_V.
\end{align}

Now, let $V_\lambda$ denote the 
irreducible representation of $G$ 
with highest weight $\lambda$. 
The {\em saturation of $\lambda$},
denoted by $\varSigma_\lambda$,  
is the set of all dominant weights
that are less than or equal to $\lambda$,
$$
\varSigma_\lambda := \{ \mu :\ 
\mu \text{ is dominant and }\ \mu \leq \lambda \}.
$$
Let $V_{\varSigma_\lambda}$ denote the representation 
$\oplus_{\mu \in \varSigma} V_\mu$, 
and let $\mc{M}_{\lambda}$
denote the reductive monoid
defined by $V_{\varSigma_\lambda}$ as in 
in (\ref{A:notation}). 
In a similar manner, we will denote 
$M_{V_{\lambda}}$ by $M_\lambda$. 
(These are special cases of the 
``multi-lined closure'' construction of 
Li and Putcha in~\cite{LiPutcha}.)
Clearly, both of the monoids $M_\lambda$ and 
$\mc{M}_\lambda$ are reductive. 

In~\cite{DeConcini}, De Concini analyzed 
the geometric properties of $\mc{M}_{\lambda}$
in relation with that of $M_\lambda$, and 
he proved the following theorem.
\begin{theorem}[DeConcini]
\begin{enumerate}
\item $\mc{M}_\lambda$ is a normal 
variety with rational singularities. 
\item $\mc{M}_\lambda$ is the normalization 
of $M_\lambda$. 
\item $\mc{M}_\lambda$ and 
$M_\lambda$ are equal if and only if 
$\lambda$ is minuscule, that is to say, 
$\varSigma_\lambda = \{\lambda \}$.
\end{enumerate}
\end{theorem}

We finish our paper with a theorem whose 
proof will be given somewhere else.

\begin{theorem}
The reductive monoid $M_\lambda$, 
hence its normalization $\mc{M}_\lambda$ have an 
$\R$-form if and only if there exists an involutory
automorphism $\theta$ of the 
reductive unit-group $G$ of $M_\lambda$
such that $\theta^* \lambda = -\lambda$.
\end{theorem}


\bibliographystyle{plain}
\bibliography{referenc}

\begin{thebibliography}{10}

\bibitem{Borel}
Armand Borel.
\newblock {\em Linear algebraic groups}, volume 126 of {\em Graduate Texts in
  Mathematics}.
\newblock Springer-Verlag, New York, second edition, 1991.

\bibitem{BraviPezzini16}
Paolo Bravi and Guido Pezzini.
\newblock Primitive wonderful varieties.
\newblock {\em Math. Z.}, 282(3-4):1067--1096, 2016.

\bibitem{Brion:Quelques}
Michel Brion.
\newblock Quelques propri\'et\'es des espaces homog\`enes sph\'eriques.
\newblock {\em Manuscripta Math.}, 55(2):191--198, 1986.

\bibitem{Brion:Duke}
Michel Brion.
\newblock Groupe de {P}icard et nombres caract\'eristiques des vari\'et\'es
  sph\'eriques.
\newblock {\em Duke Math. J.}, 58(2):397--424, 1989.

\bibitem{Brion:FieldsLectures}
Michel Brion.
\newblock On algebraic semigroups and monoids.
\newblock In {\em Algebraic monoids, group embeddings, and algebraic
  combinatorics}, volume~71 of {\em Fields Inst. Commun.}, pages 1--54.
  Springer, New York, 2014.

\bibitem{Brion:CliffordLectures}
Michel Brion.
\newblock Some structure theorems for algebraic groups.
\newblock In {\em Algebraic groups: structure and actions}, volume~94 of {\em
  Proc. Sympos. Pure Math.}, pages 53--126. Amer. Math. Soc., Providence, RI,
  2017.

\bibitem{BrockertomDieck}
Theodor Br\"ocker and Tammo tom Dieck.
\newblock {\em Representations of compact {L}ie groups}, volume~98 of {\em
  Graduate Texts in Mathematics}.
\newblock Springer-Verlag, New York, 1995.
\newblock Translated from the German manuscript, Corrected reprint of the 1985
  translation.

\bibitem{Chevalley}
Claude Chevalley.
\newblock {\em Theory of {L}ie groups. {I}}, volume~8 of {\em Princeton
  Mathematical Series}.
\newblock Princeton University Press, Princeton, NJ, 1999.
\newblock Fifteenth printing, Princeton Landmarks in Mathematics.

\bibitem{Conrad:Reductive}
Brian Conrad.
\newblock Reductive group schemes.
\newblock In {\em Autour des sch\'emas en groupes. {V}ol. {I}}, volume 42/43 of
  {\em Panor. Synth\`eses}, pages 93--444. Soc. Math. France, Paris, 2014.

\bibitem{Danilov}
Vladimir~Ivanovich Danilov.
\newblock The geometry of toric varieties.
\newblock {\em Uspekhi Mat. Nauk}, 33(2(200)):85--134, 247, 1978.

\bibitem{DeConcini}
Corrado De~Concini.
\newblock Normality and non normality of certain semigroups and orbit closures.
\newblock In {\em Algebraic transformation groups and algebraic varieties},
  volume 132 of {\em Encyclopaedia Math. Sci.}, pages 15--35. Springer, Berlin,
  2004.

\bibitem{Demazure}
Michel Demazure.
\newblock Sous-groupes alg\'ebriques de rang maximum du groupe de {C}remona.
\newblock {\em Ann. Sci. \'Ecole Norm. Sup. (4)}, 3:507--588, 1970.

\bibitem{DemazureGabriel}
Michel Demazure and Peter Gabriel.
\newblock {\em Introduction to algebraic geometry and algebraic groups},
  volume~39 of {\em North-Holland Mathematics Studies}.
\newblock North-Holland Publishing Co., Amsterdam-New York, 1980.
\newblock Translated from the French by J. Bell.

\bibitem{Helgason:GGA}
Sigurdur Helgason.
\newblock {\em Groups and geometric analysis}, volume~83 of {\em Mathematical
  Surveys and Monographs}.
\newblock American Mathematical Society, Providence, RI, 2000.
\newblock Integral geometry, invariant differential operators, and spherical
  functions, Corrected reprint of the 1984 original.

\bibitem{Huruguen}
Mathieu Huruguen.
\newblock Toric varieties and spherical embeddings over an arbitrary field.
\newblock {\em J. Algebra}, 342:212--234, 2011.

\bibitem{Knop89}
Friedrich Knop.
\newblock The {L}una-{V}ust theory of spherical embeddings.
\newblock In {\em Proceedings of the {H}yderabad {C}onference on {A}lgebraic
  {G}roups ({H}yderabad, 1989)}, pages 225--249. Manoj Prakashan, Madras, 1991.

\bibitem{Knutson}
Donald Knutson.
\newblock {\em Algebraic spaces}.
\newblock Lecture Notes in Mathematics, Vol. 203. Springer-Verlag, Berlin-New
  York, 1971.

\bibitem{Legendre}
Adrien~Marie Legendre.
\newblock Recherches sur l'attraction des spheroides homogenes.
\newblock {\em M{\'e}moires de math{\'e}matique et de physique : pr{\'e}s.
  {\`a} l'Acad{\'e}mie Royale des Sciences, par divers savans, et l{\^u}s dans
  ses assembl{\'e}es}, 1785:411 -- 434, 2007.

\bibitem{LiPutcha}
Zhuo Li and Mohan Putcha.
\newblock Types of reductive monoids.
\newblock {\em J. Algebra}, 221(1):102--116, 1999.

\bibitem{LunaVust}
Domingo Luna and Thierry Vust.
\newblock Plongements d'espaces homog\`enes.
\newblock {\em Comment. Math. Helv.}, 58(2):186--245, 1983.

\bibitem{Mumford}
David Mumford.
\newblock {\em Abelian varieties}, volume~5 of {\em Tata Institute of
  Fundamental Research Studies in Mathematics}.
\newblock Published for the Tata Institute of Fundamental Research, Bombay; by
  Hindustan Book Agency, New Delhi, 2008.
\newblock With appendices by C. P. Ramanujam and Yuri Manin, Corrected reprint
  of the second (1974) edition.

\bibitem{OnishchikVinberg}
Arkadii~L'vovich Onishchik and \`Ernest~Borisovich Vinberg.
\newblock {\em Lie groups and algebraic groups}.
\newblock Springer Series in Soviet Mathematics. Springer-Verlag, Berlin, 1990.
\newblock Translated from the Russian and with a preface by D. A. Leites.

\bibitem{Ono}
Takashi Ono.
\newblock On the field of definition of {B}orel subgroups of semi-simple
  algebraic groups.
\newblock {\em J. Math. Soc. Japan}, 15:392--395, 1963.

\bibitem{Popov:Contractions}
Vladimir~Leonidovich Popov.
\newblock Contractions of actions of reductive algebraic groups.
\newblock {\em Mat. Sb. (N.S.)}, 130(172)(3):310--334, 431, 1986.

\bibitem{Putcha}
Mohan~S. Putcha.
\newblock {\em Linear algebraic monoids}, volume 133 of {\em London
  Mathematical Society Lecture Note Series}.
\newblock Cambridge University Press, Cambridge, 1988.

\bibitem{Renner:vonNeumann}
Lex~E. Renner.
\newblock Reductive monoids are von {N}eumann regular.
\newblock {\em J. Algebra}, 93(2):237--245, 1985.

\bibitem{Renner:ClassificationVarieties}
Lex~E. Renner.
\newblock Classification of semisimple varieties.
\newblock {\em J. Algebra}, 122(2):275--287, 1989.

\bibitem{Renner}
Lex~E. Renner.
\newblock {\em Linear algebraic monoids}, volume 134 of {\em Encyclopaedia of
  Mathematical Sciences}.
\newblock Springer-Verlag, Berlin, 2005.
\newblock Invariant Theory and Algebraic Transformation Groups, V.

\bibitem{Rittatore98}
Alvaro Rittatore.
\newblock Algebraic monoids and group embeddings.
\newblock {\em Transform. Groups}, 3(4):375--396, 1998.

\bibitem{StacksProject}
The {Stacks Project Authors}.
\newblock Stacks project.
\newblock \url{http://stacks.math.columbia.edu}, 2017.

\bibitem{Timashev}
Dmitry~A. Timashev.
\newblock {\em Homogeneous spaces and equivariant embeddings}, volume 138 of
  {\em Encyclopaedia of Mathematical Sciences}.
\newblock Springer, Heidelberg, 2011.
\newblock Invariant Theory and Algebraic Transformation Groups, 8.

\bibitem{Vinberg:Complexity}
\`Ernest~Borisovich Vinberg.
\newblock Complexity of actions of reductive groups.
\newblock {\em Funktsional. Anal. i Prilozhen.}, 20(1):1--13, 96, 1986.

\bibitem{Vinberg:ReductiveSemigroups}
\`Ernest~Borisovich Vinberg.
\newblock On reductive algebraic semigroups.
\newblock In {\em Lie groups and {L}ie algebras: {E}. {B}. {D}ynkin's
  {S}eminar}, volume 169 of {\em Amer. Math. Soc. Transl. Ser. 2}, pages
  145--182. Amer. Math. Soc., Providence, RI, 1995.

\bibitem{Wedhorn}
Torsten Wedhorn.
\newblock Spherical spaces.
\newblock {\em Annales de l'Institut Fourier}, 68(1):229--256, 2018.

\end{thebibliography}

\end{document}